\documentclass[12pt]{amsart}
\usepackage{amsmath, amssymb}

\newfont {\cyr} {wncyr10}

\newtheorem{theorem}{Theorem}[section]

\newtheorem{lemma}[theorem]{Lemma}

\newcounter{claim}[theorem]
\renewcommand{\theclaim}{\noindent{(\thesection.\arabic{theorem}.\arabic{claim})}}

\newenvironment{claim}[1][\theclaim]{\begin{trivlist}\refstepcounter{claim}
\item[\hskip \labelsep {\bfseries #1}]}{\end{trivlist}}
\def\qedc {$\hfill \spadesuit$\bigskip}

\newenvironment{matrix4}{\left( \begin{smallmatrix}}
{\end{smallmatrix} \right)}

\def \G {\Gamma}

\def \PGL {{\rm PGL}}\def \SD {{\rm SDih}}
\def \GO {\mathrm{GO}}
\def \a {\alpha}

\def \l {\lambda}

\def \Sym {\mbox {\rm Sym}}

\def \SU {\mbox {\rm SU}}

\def \dim {\mbox {\rm dim}}

\def \CO     {\mbox {\rm CO}}\def \SO     {\mbox {\rm SO}}
\def \GSp {\mbox {\rm GSp}}

\def \Dih {{\mathrm {Dih}}}
\def \SDih {{\mathrm {SDih}}}

\def \signal {\mbox{ {{\cyr I}}}}

\def \SL {\hbox {\rm SL}}

\def \GL {\mbox {\rm GL}}

\def \syl {\hbox {\rm Syl}}\def \Syl {\hbox {\rm Syl}}

\def \ov {\overline}

\def \Aut{ \mathrm {Aut}}

\def \Out{\mbox {\rm Out}}

\def \Fi {\mbox {\rm Fi}}

\def \J{\mbox {\rm J}}

\def \Ly{\mbox {\rm Ly}}

\def \B{\mbox {\rm B}}

\def \M{\mbox {\rm M}}

\def \HN{\mbox {\rm HN}}

\def \HS{\mbox {\rm HS}}

\def \Th{\mbox {\rm Th}}

\def \ON {\mbox {\rm O'N}}

\def \Co {\mbox {\rm Co}}

\def \Ru {\mbox {\rm Ru}}

\def \Suz{\mbox {\rm Suz}}

\def \McL{\mbox {\rm McL}}

\def \He {\mbox {\rm He}}

\def \B  {\mbox {{\rm B}}}

\def \U {\mbox {{\rm SU}}}

\def \F {\mbox {{\rm F}}}

\def \PSp  {\mathrm {PSp}}

\def \Alt {\mbox {{\rm Alt}}}

\def \GG {\mbox {{\rm G}}}

\def \GF {{\mathrm {GF}}}

\def \Q {\mbox {{\rm Q}}}

\def \PSL {\mbox {{\rm PSL}}}

\def \Sp {\mbox {{\rm Sp}}}

\def \G {\Gamma}

\def \a {\alpha}

\def \l {\lambda}

\def \F4pic {

\begin{center}

\unitlength10mm

\begin {picture}(15,1)\thinlines

\put (0.0,.3){${\rm F}_4$}

\put (1.5,.5){\circle {.1}}

\put (2.5,.5){\circle {.1}}

 \put (3.5,.5){\circle {.1}}

\put (4.5,.5){\circle {.1}}

\put (1.55, .5){\line (1,0) {.9}}

 \put (2.5,.55){\line (1,0) {1.0}}

\put (2.5,.45){\line (1,0){1.0}}

\put (3.55,.5){\line (1,0) {.9}}

\put (2.9,.425){$>$}

\put (1.5,0.2){\scriptsize {$\a_1$}}

\put (2.5,0.2){\scriptsize {$\a_2$}}

\put (3.5,0.2){\scriptsize {$\a_3$}}

\put (4.5,0.2) {\scriptsize {$\a_4$}}

\end {picture}

\end{center}}

\def \G2pic {

\begin{center}

\unitlength10mm

\begin {picture}(15,1)\thinlines

\put (0.0,.3){$\GG_2$}

 \put (1.5,.5){\circle {.1}}

\put (2.5,.5){\circle {.1}}

 \put (1.5,.55){\line (1,0) {1.0}}

\put (1.5,.45){\line (1,0){1.0}}

 \put (1.55,.5){\line (1,0) {.9}}

\put (1.9,.422){$>$}

 \put (1.5,0.2){\scriptsize {$\a_1$}}

\put (2.5,0.2) {\scriptsize {$\a_2$}}

\end {picture}

\end{center}}

\begin{document}
\title  {A $3$-local characterization of $\Co_2$}
 \author{Christopher Parker}
\author{Peter Rowley}
\address{Christopher Parker\\
School of Mathematics\\
University of Birmingham\\
Edgbaston\\
Birmingham B15 2TT\\
United Kingdom} \email{c.w.parker@bham.ac.uk}
\address{Peter Rowley\\
School of Mathematics\\
University of Manchester\\
Oxford Road\\
M13 6PL\\
United Kingdom} \email{peter.j.rowley@manchester.ac.uk}

\date{\today}

\maketitle \pagestyle{myheadings}

\markright{{\sc Characterization of $\Co_2$}} \markleft{{\sc
Christopher Parker and Peter Rowley}}

\begin{abstract}
Conway's second largest simple group, $\Co_2$, is characterized by
the centralizer of an element of order $3$ and certain fusion data.
\end{abstract}

\section{Introduction}\label{Introduction}

The vistas revealed  by Goldschmidt in \cite{GO1} inspired many
investigations of amalgams, particularly in their application to
finite groups  and their geometries. One such was the fundamental
work of Delgado and Stellmacher \cite{delstell} in which weak $BN$
pairs were classified. Later Parker and Rowley \cite{WeakBN}
determined the finite local characteristic $p$ completions of weak
$BN$ pairs (when $p$ is odd and excluding the amalgams of type
$\PSL_3(p)$). However a number of exceptional configurations when
$p\in \{3,5,7\}$ required further attention--all but one of them
have been addressed in Parker and Rowley \cite{PRLy}, \cite{PRMcL},
Parker \cite{P} and Parker and Weidorn \cite{Monster}. The last one
is run to ground here in our main result which gives a characterization of
Conway's second largest simple group, $\Co_2$.

\begin{theorem}\label{Th1} Suppose that $G$ is a finite group,
$S \in \Syl_3(G)$, $Z =Z(S)$ and $C = C_G(Z)$.  Assume that $O_3(C)$
is extraspecial of order $3^5$, $O_{2}(C/O_3(C))$ is extraspecial of
order $2^5$ and $C/O_{3,2}(C) \cong \Alt(5)$.  If $Z$ is not weakly
closed in $S$ with respect to $G$, then
 $G$ is isomorphic to $\Co_2$.
\end{theorem}

The hypothesis on the structure of $C$ in Theorem~\ref{Th1} amounts
to saying that $C$ has shape $3^{1+4}.2^{1+4}.\Alt(5)$. Note that no
assertion about the types of extension is included and the
extraspecial groups could have either $+$- or $-$-type. We remark,
as may be seen from  \cite{Wilson} or \cite{Atlas}, that $\Co_2$ actually satisfies
the hypothesis of Theorem~\ref{Th1}.  As a consequence of
Theorem~\ref{Th1} and earlier work on the exceptional cases arising
in \cite{WeakBN}, we can now see that part (ii) of \cite[Theorem
1.5]{WeakBN} does not occur. Theorem~\ref{Th1} investigates a more
general configuration than required to settle \cite[Theorem 1.5 (ii)
(c)]{WeakBN}. Though not immediately apparent, this configuration
rather quickly gives rise to a subgroup $M_*$ of shape
$3^4.\Omega_4^-(3) \cong 3^4.\Alt(6)$. This particular subgroup makes appearances in other
simple groups such as $\U_4(3)$, $\mathrm {PSU}_6(2)$ and $\McL$ and is the
root cause of the exceptional possibilities itemized in
\cite[Theorem 1.5 (ii)(a), (b) and (c)]{WeakBN}.

A number of the  sporadic simple groups have been characterized in
terms of  $3$-local data. The earliest being a characterization of
$\J_1$ by Higman \cite[Theorem 12]{Higman}. In \cite{ON}, O'Nan
determined the finite simple groups having an elementary abelian
subgroup $P$ of order $3^2$ such that for $x\in P^\#$,
$C_G(x)/\langle x \rangle$ is isomorphic to $\PSL_2(q)$, $\PGL_2(q)$
or $\mathrm P\Sigma\mathrm L_2(q)$ ($q$ odd). Thereby he
characterized the sporadic simple groups $\M_{22}$, $\M_{23}$,
$\M_{24}$, $\J_2$, $\HS$ and $\Ru$. For the remaining Janko groups,
$3$-local identifications for $\J_3$ were obtained first by Durakov
\cite{Durakov} and later by Aschbacher \cite{AschbacherJ3}, and for
$\J_4$ by Stroth \cite{Stroth}, Stafford \cite{Stafford} and
G\"ulo\u glu \cite{G}. The groups $\ON$ and $\He$ were dealt with,
respectively, by Il\'{}inyh \cite{AP} and Borovik \cite{Borovik}.
All of these results were obtained prior to 1990. Recently there has
been a resurgence of interest and activity in $3$-local
characterizations of finite simple groups partly prompted by the
revision project concerning groups of local characteristic $p$ (see,
for example, \cite{MSS}). The sporadic simple groups studied in this
renaissance period are $\Co_3$ (Korchagina, Parker and Rowley
\cite{CPI}), $\Fi_{22}$ (Parker \cite{P}), $\McL$ (Parker and
Rowley \cite{PRMcL}), $\M_{12}$ (Astill and Parker \cite{AstillParker}), $\Th$ (Fowler
\cite{Fowler}), and $\Co_1$, $\Fi_{24}^\prime$,
(Salarian \cite{SalerianCo1,SalerianF24} ) and  $\M$ (Salarian and Stroth \cite{SalerianM}).

With a few exceptions, to date, characterization results for finite
groups in terms of $3$-local data ultimately rely upon identifying
the target group(s) via $2$-local information. This is the case
here, F. Smith's Theorem \cite{Smith} providing the final
identification. Thus most of this paper is spent manoeuvering into a
position where we can use this result. We begin in Section~2 giving
background results-- F. Smith's Theorem appearing as
Theorem~\ref{Smith}. Another characterization result appearing in
Theorem~\ref{PrinceThm}, due to Prince, is employed in
Lemma~\ref{Prince2}. Lemma~\ref{Prince2}, which is the bridge to the
$2$-local structure of  $G$ ($G$ as in Theorem~\ref{Th1}), states
that $N_G(B) \cong \Sym(3)\times\Aut(\U_4(2))$ for a certain
subgroup $B$ of $G$ of order $3$. In $N_G(B)$ there is an involution
$t$ inverting $B$ and centralizing $O^3(C_G(B))\cong \Aut(\U_4(2))$.
Not only does this lemma fill out our knowledge of the $3$-local
subgroups but it also gives  us a toehold in $C_G(t)$. After
Lemmas~\ref{cen3psp43}--\ref{extraauto}, results which play minor
supporting roles,  we present Lemmas~\ref{quadratic form}, \ref{GO4} and \ref{GO4minus} which are pivotal for the identification
of the normalizer of $J$, the Thompson subgroup of  $S$, $S \in \Syl_3(G)$. It turns out that $J$ is elementary abelian of order $3^4$ and these lemmas allow us to assert in Lemma~\ref{LMStruct} that $N_G(J)/J \cong \CO_4^-(3)$, the group of all similitudes of a non-degenerate orthogonal form of $-$-type in dimension $4$. This opens the way for us to use facts about the action of this group on $J$. The pertinent facts are listed in Lemma~\ref{Modcalc}.
This plays an
important role in Lemma~\ref{JSig2} where we show that
$3'$-signalizers for $J$ are trivial.  Various properties of groups of shape
$2^{1+4}.\Alt(5)$ are given in Lemmas~\ref{Laction}, \ref{Laction-1}
and \ref{Laction2}. These results will be applied to bring the
structure of $C_G(Z)$  into sharper focus, where $Z = Z(S)$. We
conclude Section~2 with Lemmas~\ref{SP62facts} and \ref{NOTFF} which
concern the spin module for $\Sp_6(2)$, followed by an elementary
result on $\Aut(\U_4(2))$ in Lemma~\ref{88}.

The main result of Section~3, Theorem~\ref{closed}, anticipates the
end game in our analysis of $C_G(t)$, $t$ being the involution
mentioned earlier. In fact, Theorem~\ref{closed} will be applied to
$C_G(t)/\langle t\rangle$.

Section~4 sees us start the proof of Theorem~\ref{Th1}. After
Lemma~\ref{lem1} in which the structure of $C_G(Z)$ is examined
(where $Z=Z(S)$, $S\in\Syl_3(G)$), Lemmas~\ref{inv1} and \ref{inv2}
look at centralizers and commutators of certain involutions in
$C_G(Z)$. In Lemmas~\ref{SACTION}, \ref{New2} and \ref{centralizers}
it is $S$ and its subgroups that mostly occupy our attention. Two
subgroups of $S$ that will play central roles in the proof of
Theorem~\ref{Th1} are $Z$ and $J=C_S([Q,S])$ where $Q=O_3(N_G(Z))$.
In Lemmas~\ref{New1} and \ref{New2} we learn that $J$ is the Thompson subgroup of
$S$, $J$ is elementary abelian of order $3^4$ and that all
$G$-conjugates of $Z$ in $S$ are trapped inside $J$. Another
important subgroup of $S$, namely  $B$, along with the involution
$t$, already noted earlier, make their entrance after
Lemma~\ref{LMStruct}. In the latter part of Section~4, our attention
moves on to $N_G(Z)$, resulting in structural information about this
subgroup in Lemmas~\ref{M0} and \ref{JOrbs}. Drawing
upon the results in Section 4, in Section~5 we determine the
structure of $N_G(B)$. Our last section brings to bear all the
earlier results on $C_G(t)$ eventually yielding that $C_G(t)/\langle
t\rangle$ satisfies the hypotheses of Theorem~\ref{closed}. Then
using Theorem~\ref{closed} we rapidly obtain the hypotheses of
Theorem~\ref{Smith}, whence we deduce that $G \cong \Co_2$.

We follow the {\sc Atlas} \cite{Atlas}
notation and conventions there with a number of variations which we
now mention or hope are self explanatory. We shall use $\Sym(n)$ and $\Alt(n)$ to denote,
respectively, the symmetric and alternating groups of degree $n$ and
$\Dih(n)$, $\Q(n)$ and $\SD(n)$, respectively, to stand for the
dihedral group, quaternion group and semidihedral group of order
$n$. Finally $X\sim Y$ where $X$ and $Y$ are groups will indicate
that $X$ and $Y$ have the same shape.

The remainder of our notation is standard as given, for example, in
\cite{Aschbacher} and \cite{KurzStell}.

\noindent {\bf Acknowledgement.}  This paper is the fruit of a visit
to the Mathematisches Forschungsinstitut Oberwolfach as part of the
Research in Pairs Programme, 29th April--12 May, 2007. The authors
wish to thank the institute and its staff for the pleasant and
stimulating environment that they provided. It is also a pleasure to thank Ulrich Meierfrankenfeld for his comments on an earlier version of this paper.

\section{Preliminary Results}

\begin{theorem}[F. Smith]\label{Smith} Suppose that $X$ is a finite group with $Z(X)=
O_{2'}(X)=1$, and $Y$ is the centralizer of an involution in $X$. If
$Y /O_2(Y) \cong \Sp_6(2)$ and $O_2(Y)$ is a non-abelian group of
order $2^9$ such that the elements of order $5$ in $Y$ act fixed
point freely on $O_2(Y)/Z(O_2(Y))$, then $X$ is isomorphic to
$\Co_2$.
\end{theorem}
\begin{proof} See \cite{Smith}.
\end{proof}

\begin{theorem}[A. Prince] \label{PrinceThm}
Suppose that $Y$ is isomorphic to the centralizer of $3$-central element of
order $3$ in $\PSp_4(3)$ and that $X$ is a finite group with a non-trivial
element $d$ such that $C_X(d)\cong Y$. Let $P \in \Syl_3(C_X(d))$
and $E$ be the elementary abelian subgroup of $P$ of order $27$. If
$E$ does not normalize any non-trivial $3^\prime$-subgroup of $X$
and $d$ is $X$-conjugate to its inverse, then either
\begin{enumerate}
\item
$|X:C_X(d)| =2$;
\item $X$ is isomorphic to $\Aut(\U_4(2))$; or
\item $X$ is isomorphic to $\Sp_6(2)$.
\end{enumerate}
\end{theorem}

\begin{proof} See \cite[Theorem 2]{Prince}. \end{proof}

\begin{lemma} \label{cen3psp43}Suppose that $X$ is a group of shape $3^{1+2}_+.\SL_2(3)$,  $O_2(X)=1$
 and
a Sylow $3$-subgroup of $X$ contains an elementary abelian subgroup
of order $3^3$. Then $X$ is isomorphic to the centralizer of a
non-trivial $3$-central element in $\PSp_4(3)$.
\end{lemma}
\begin{proof} See \cite[Lemma~6]{P}.\end{proof}

We will also use the following variation of Lemma~\ref{cen3psp43}.
\begin{lemma}
\label{psp43new}Suppose that $X$ is a group of shape
$3^{1+2}_+.\SL_2(3)$, $O_2(X)=1$ and the Sylow $3$-subgroups of a
centralizer of an involution in $X$ are elementary abelian.
 Then $X$ is isomorphic to the centralizer of a
non-trivial $3$-central element in $\PSp_4(3)$.
\end{lemma}

\begin{proof} Let $S\in \Syl_3(X)$, $R = O_3(X)$, and $F\le R$ be a normal subgroup of $S$ of order
$9$. Let $N= N_X(S)$. If $F$ is not normal in $N$, then there exists
$n \in N$ such that $R= F^nF$. But then $S$ centralizes $FF^n/Z(R)=
R/Z(R)$ and so $C_X(R/Z(R)) > R$ and this contradicts $O_2(X) \neq
1$. Hence $F$ is normal in $N$. Let $E=C_S(F)(=C_N(F))$. Then $E$ is
abelian of order $27$. Let $u$ be an involution in $N$. Then $u$
normalizes $E$
 and, as $[S,u]\le R$, $C_E(u)\not \le R$. Therefore $E= C_E(u)F$.
 Since $F$ and
 $C_E(u)$ are elementary abelian by hypothesis,  $E$ is elementary abelian of order $3^3$. Hence
Lemma~\ref{cen3psp43} applies and yields the result.
\end{proof}

\begin{lemma}\label{fusion} Suppose that $p$ is a prime, $X$ is a finite group and $P\in \Syl_p(X)$.
If $x,y \in Z(J(P))$ are $X$-conjugate, then $x$ and $y$ are
$N_X(J(P))$-conjugate.
\end{lemma}

\begin{proof} See \cite[37.6]{Aschbacher}.\end{proof}

\begin{lemma}\label{wcl1} Suppose that $p$ is a prime, $X$ is a finite group and $ P \in
\syl_p(X)$. If $R \le P$ is not weakly closed in $P$ with respect to
$X$, then there exists $x \in X$ such that $R\neq R^x$ and $R$ and
$R^x$ normalize each other.
\end{lemma}

\begin{proof} Suppose that $R$ is not normal in $P$. Let
$N= N_P(R)$ and $M= N_P(N)$.  Then $M>N$. Choose $x\in M\setminus
N$. Then $R\neq R^x$ and, as $R$ and $R^x$ are both normal in $N$,
we obtain the lemma. Hence we may assume that $R$ is normal in $P$.
Since $R$ is not weakly closed in $P$ with respect to $X$, there
exists $y \in X$ such that $R^y \neq R$ and $R^y \le P$. If $R^y$ is
normal in $P$, then $R$ and $R^y$ normalize each other and we take
$x=y$. Otherwise, repeating the argument as for $R$, we find $z\in
P$ such that $R^y$ and $R^{yz}$ normalize each other. Taking $x=
yzy^{-1}$ completes the proof of the lemma.
\end{proof}

\begin{lemma}\label{orbits} Suppose that $X$ is a finite group, $x \in X$ an
involution of $X$ and $V$  an elementary abelian normal $2$-subgroup
of $X$. Set $C= C_X(x)$. Then the map $(vx)^{VC} \mapsto (v[V,x])^C$
is a bijection between $VC$-orbits of the involutions in the coset
$Vx$ and the $C$-orbits of the elements of $C_V(x)/[V,x]$.
Furthermore, for $vx$ an involution in $Vx$, $|(vx)^{VC}|=
|(v[V,x])^C|.|[V,x]|$.
\end{lemma}

\begin{proof} The given map is easily checked to be a bijection.
\end{proof}

\begin{lemma}\label{extraauto} Suppose that $Q$ is an extraspecial $p$-group and
$\alpha \in \Aut(Q)$. If $A$ is a maximal abelian subgroup of $Q$
and $[A,\alpha]=1$, then $\alpha$ is a $p$-element.
\end{lemma}

\begin{proof} The Three Subgroup Lemma implies that $[Q, \alpha]\le
A$. Then $ [Q,\alpha,\alpha] \le [A,\alpha]=1$ and so $\alpha$ is a
$p$-element.
\end{proof}

%

%
%

When we are studying signalizers in Lemma~\ref{signal1}, we shall call on the following lemma repeatedly.

\begin{lemma}\label{const} Suppose that $p$ is a prime, $X$ is a group and $P$ is a $p$-subgroup of $X$. If $U \le O_{p'}(N_X(P))$ and $U$ and $P$ are contained in some soluble subgroup $Y$ of $K$, then $U\le O_{p'}(Y)$.
\end{lemma}

\begin{proof} See \cite[8.2.13, pg. 190]{KurzStell}. \end{proof}

The proof of the next lemma is taken from \cite[Lemma 1]{King}.

\begin{lemma}\label{quadratic form}
Suppose that $F$ is a field, $V$ is a finite dimensional vector space over $F$ and  $X= \GL(V)$. Assume that $q$ is a quadratic form  of Witt index at least $1$ and with non-degenerate associated bilinear form $f$, where, for $v,w \in V$, $f(v,w) = q(v+w)-q(v)- q(w)$.
Let $\mathcal S$ be the set of singular 1-dimensional subspaces of $V$ with respect to $q$. Then the stabilizer in $X$ of $\mathcal S$ preserves $q$ up to similarity.
\end{lemma}

\begin{proof} Let $Y$ be the subgroup of $X$ preserving $q$ up to similarity. Assume that $g \in X$ stabilizes $\mathcal S$ and select $\langle x\rangle, \langle y\rangle \in \mathcal S$ such that $f(x,y)=1$. Then $W = \langle x,y \rangle$ is a hyperbolic plane. Since $g$ preserves $\mathcal S$,  $Wg$ is also a hyperbolic plane. By Witt's Lemma \cite[pg. 81]{Aschbacher}, $Y$ contains an element mapping $Wg$ to $W$ which also maps $\langle x g\rangle$ to $\langle x \rangle $ and $\langle yg \rangle $ to $\langle y\rangle$. Hence multiplying $g$ by a suitable element of $Y $ we may  assume that $xg= x$ and $yg= \lambda y$ for some $\lambda \in F$. Let $z \in W^\perp$ and set $U=\langle x,z\rangle g = \langle x, zg\rangle$.  Since $f(x,z)=0=q(x)$, for $\mu \in F$ we have $q(\mu x +z)= q(z)$. So either every one-space of $\langle x,z\rangle$ is singular, or  $q(z)\neq 0$, and  $\langle x\rangle$ is the only singular one-space in $\langle x,z\rangle$. Since $g$ stabilizes $\mathcal S$, it follows that either $U$ is totally singular, or  $\langle x\rangle$ is the only singular  one-space contained in $U$. Hence, in either case,  $zg \in x^\perp$. A similar argument also shows that $zg \in y^\perp$.
 Hence $zg \in W^\perp$. Since $z \in W^\perp$, $z+x-q(x)y$ is a singular vector and thus, as $g$ maps singular vectors to singular vectors,
 $zg+x-q(x)\lambda y$ is also a singular vector. Now, using $z^g \in W^\perp$, we obtain $q(zg)= \lambda q(z)$.
 Because $V= W\oplus W^\perp$ we then conclude that  $q(vg)= \l q(v)$ for all $v \in V$ and so $g \in Y$ as claimed.
\end{proof}

\begin{lemma}\label{GO4} Suppose that $p$ is an odd prime, $X = \GL_4(p)$ and $V$ is the natural $\GF(p)X$-module. Let $A =\langle a, b\rangle\le X$ be elementary abelian of order $p^2$ and assume that $[V,a] = C_V(b)$ and $[V,b]= C_V(a)$ are distinct and of dimension $2$.
Let $v \in V\setminus [V,A]$. Then  $A$ leaves invariant a non-degenerate quadratic form with respect to which  $v$ is a singular vector and $C_V(A)$ is a singular one-space. In particular, $X$ contains exactly two conjugacy classes of subgroups such as $A$,  one being conjugate to a Sylow $p$-subgroup of $\GO_4^+(p)$ and the other to a Sylow $p$-subgroup of $\GO_4^-(p)$.
\end{lemma}

\begin{proof} Since $A$ is a $p$-group,  $C_V(A)= C_V(a)\cap C_V(b)$ has dimension $1$ and $[V,A]= [V,a]+[V,b]$ has dimension $3$. Also note that $[V,A]/C_V(A) = C_{V/C_V(A)}(a) = C_{V/C_V(A)}(b)$. We have $va \in v+[V,a]$ but $[v,a] \not \in C_V(A)$. Hence $va = v+ w$ where $w\in [V,a]\setminus C_V(A)$. Similarly $vb = v+ x$ where $x \in [V,b]\setminus C_V(A)$. Also $wa = w + y$ for some $y \in C_V(A)^\#$ and then $xb = x+\lambda y$ for some $\lambda \in \GF(p)^\#$. Take $\{v,w,x,y\}$ as an ordered basis of $V$. With respect to this basis  $a$ corresponds to the matrix
$\begin{matrix4}1&1&0&0\\0&1&0&1\\0&0&1&0\\0&0&0&1\end{matrix4}$ and $b$ corresponds to $\begin{matrix4}1&0&1&0\\0&1&0&0\\0&0&1&\lambda\\0&0&0&1\end{matrix4}$.
 Let $f$ be the symmetric bilinear form on $V$ which has matrix $Y=\begin{matrix4}0&-1/2&-\lambda/2&-1\\-1/2&1&0&0\\-\lambda/2&0&\lambda&0\\-1&0&0&0\end{matrix4}$. Then $a$ and $b$ preserve $f$ and, since $\det Y = -\lambda\not=0$, $f$ is non-degenerate. Obviously $v$ is a singular vector and $C_V(A)$ is a singular one-space with respect to $f$. Since the Sylow $p$-subgroups of $\GO_4^\pm(p)$ have order $p^2$, the lemma is proven.
\end{proof}

\begin{lemma}\label{GO4minus} Suppose that $p$ is an odd prime, $X = \GL_4(p)$,  $A \le X$ is elementary abelian of order $p^2$, $V$ is the natural $\GF(p)X$-module and $v \in V\setminus [V,A]$.
Assume that no element of $A$ acts quadratically on $V$ and that $\dim [V,a] =2$ for all $a \in A^\#$. Then $A$ preserves a quadratic form of $-$-type which has singular $1$-spaces $\{C_V(A)\}\cup\{\langle v\rangle a \mid a \in A\}$.
\end{lemma}

\begin{proof} Suppose that $a \in A^\#$. Then $\dim [V,a] = \dim C_V(a)=2$, $[V,A]= [V,a]+C_V(a)$ and $C_V(A)=[V,a] \cap C_V(a)$. Since $\dim C_V(A)=1$, for $a, b \in A^\#$ with $\langle a\rangle \neq \langle b\rangle$, $C_V(a) \neq C_V(b)$ and therefore, fixing $a \in A^\#$, there exists  unique cyclic subgroup $\langle b \rangle \le A$ such that $C_V(b)= [V,a]$. Now, as $a$ and $b$ commute, $[V,b,a]=[V,a,b]$ and therefore $[V,a]= C_V(b)$. We now fix $a$ and $b$ as generators of $A$ and apply Lemma~\ref{GO4}. This shows us that $A$ preserves a non-degenerate quadratic form $q$ and that $q(v)=0$. Since the Sylow $p$-subgroup of $\GO_4^+(p)$ contains elements which act quadratically, we infer that $q$ has $-$-type. In particular, $V$ has $p^2+1$
singular vectors with respect to $q$. Since   $\{C_V(A)\}\cup\{\langle v\rangle x \mid x \in A\}$ are all singular  and $|\{\langle v\rangle x \mid x \in A\}|=p^2$, the result follows.
\end{proof}
\begin{lemma}\label{Modcalc}
Suppose that $X = \Omega_4^-(3)$ and let $V$ be the natural $\GF(3)X$-module.
Then
the following hold.
\begin{enumerate}
\item $X$ has three orbits  $\mathcal O_0$,  $\mathcal O_1$ and  $\mathcal O_2$ on the one-dimensional subspaces of $V$. The set
 $\mathcal O_0$ consists of singular one-spaces, while $\mathcal O_1$ and $\mathcal O_2$  consist of non-singular one-spaces.
 Furthermore,  $|\mathcal O_0|= 10$ and $|\mathcal
O_1|=|\mathcal O_2|=15$. The stabilizers of a member of $\mathcal
O_1$ and  of a member of $\mathcal O_2$ are not conjugate in $X$.

\item If $t$ is an involution in $X$, then $\dim\;C_V(t)=2$ and $C_V(t)$ is a hyperbolic space. The subspace $C_V(t)$ contains  two subspaces from $\mathcal O_0$ and one
each from $\mathcal O_1$ and $\mathcal O_2$.  Furthermore, $C_X(t)
\cong \Dih(8)$ interchanges the two members of $\mathcal O_0$ in
$C_V(t)$ and $|C_X(t)/C_{C_X(t)}(C_V(t))|=4$.

\item If $g\in X$ has order $4$, then $C_V(g)=0$.

\item If $D\in \Syl_3(X)$, then $\dim\; C_V(D)=\dim\; V/[V,D]=1$ and
$C_V(D)\in \mathcal O_0$.
\item If $d\in X$ has order $3$, then $\dim \;C_V(d)= \dim \; [V,d]=2$ and $d$ is not quadratic on $V$.
\item If $D\in \Syl_3(X)$ and $t\in N_X(D)$ is an involution, then
$t$ centralizes $C_V(D)$ and $V/[V,D]$.
\end{enumerate}
\end{lemma}

%
%
%
%

\begin{proof}
This is an elementary calculation.
\end{proof}

\begin{lemma}\label{hyperplanes} Suppose that $X$, $V$ and $\mathcal O_0$ are as in
Lemma~\ref{Modcalc} and assume that $V_0$ is a hyperplane of $V$.
Then $V_0$ contains a member of  $\mathcal O_0$.
\end{lemma}

\begin{proof}  Every $3$-dimensional subspace of an orthogonal space contains a singular vector.
\end{proof}
%
%
%
%

\begin{lemma}\label{Laction}
Suppose that $V$ is a faithful $4$-dimensional $\GF(3)X$-module and
that  $X$ contains a normal subgroup $Y$ with $Y \sim
2^{1+4}.\Alt(5)$. Then $X$ is $2$-constrained, $O_2(X) = O_2(Y) $ is
extraspecial of $-$-type and either $X=Y$ or $X/O_2(X) \cong
\Sym(5)$.
\end{lemma}

\begin{proof} Let $Q= O_2(Y)$. Then $Q$ is normalized by $X$. Let $Z
= C_X(Q)$. Then, as $Q$ acts irreducibly on $V$ and $\GF(3)$ is a
splitting field for this action, $Z= Z(Q)$ by Schur's Lemma
\cite{Aschbacher}. It follows that $\Aut(Q)$ contains a subgroup
isomorphic to $2^4.\Alt(5)$ and so $Q$ is extraspecial of $-$-type.
Hence $\Aut(Q) \cong 2^4.\Sym(5)$  by \cite[Theorems 20.8 and
20.9]{DH} and this proves the result.
\end{proof}

\begin{lemma}\label{Laction-1} Suppose that $X \sim 2^{1+4}_-.\Alt(5)$ is $2$-constrained. Let $Q=
O_2(X)$ and $T \in \Syl_3(X)$.
\begin{enumerate}
\item If $i \in Q$ is a non-central involution, then $|i^X|= 10$ and
$C_X(i) \sim (\Q(8)\times 2). \Alt(4)$. In particular,
$C_X(i)Q/Q\cong \Alt(4)$; and
\item $C_Q(T) \cong \Dih(8)$ and $N_X(T) Q/Q\cong \Sym(3)$.
\end{enumerate}
\end{lemma}

\begin{proof} We know that $Q$ is the central product of $\Dih(8)$ and $\Q(8)$ and so it is straightforward
to calculate that there are 10 non-central involutions. They are
conjugate in pairs in $Q$ and the element of order $5$ in $X$ acts
fixed point freely on $Q/Z(Q)$. It is now easy to confirm the
details stated in (i). Since elements of order $3$ in $X$ centralize
a non-central involution and since $C_Q(T)$ is extraspecial, we get
$C_Q(T) \cong \Dih(8)$. The second part of (ii)  follows from the
Frattini Argument.
\end{proof}

\begin{lemma}\label{Laction2}
Suppose that $V$ is a faithful $4$-dimensional $\GF(3)Y$-module and
that   $Y \sim 2^{1+4}_-.\Alt(5)$.  Then the following hold.
\begin{enumerate}
\item For $v \in V^\#$, we have $C_Y(v) \cong \SL_2(3)$. In
particular, $Y$ operates transitively on $V^\#$.
\item Every element of order $2$ in $Y$ is contained in $O_2(Y)$.
\item If $T \in \Syl_3(Y)$, then $N_Y(T)/T \cong \SDih(16)$.
\end{enumerate}
\end{lemma}

\begin{proof} Let $Q= O_2(Y)$, $s \in Z(Q)^\#$ and $v \in V^\#$. Then $s$
negates $v$ and so $C_Q(v)$ is a subgroup of $Q$ which does not
contain $s$. Since $Q \cong 2^{1+4}_-$, we get that $C_Q(v)$
has order dividing $2$. Hence every orbit of $Y$ on $V$ has order
divisible by $16$. Since the elements of $Y$ of order $5$ centralize
only the zero vector, the orbits of $Y$  have length divisible by
$5$. As there are 80 non-zero vectors it follows that $Y$ acts
transitively on $V^\#$, $|C_Q(v)|=2$ and $C_Y(v)Q/Q \cong \Alt(4)$.
Since $Y$ is perfect and is isomorphic to a subgroup of $\SL_4(3)$,
the $2$-rank of $Y$ is at most $3$. By considering $\langle
s,C_Y(v)\rangle$ we see that $C_Y(v) \not \cong 2\times \Alt(4)$ and
therefore $C_Y(v)$ is isomorphic to the unique double cover of
$\Alt(4)$, namely $\SL_2(3)$. This proves (i).

Now suppose that $y \in Y\setminus Q$ has order $2$. Then as $y$ is
a noncentral involution in $Y$, $C_V(y) \not = 0$.  But then (i)
implies $y \in Q$, a contradiction.  Hence (ii) holds.

We now claim that $N_{Y}(T)/T \cong \SDih(16)$.  Since $T$ has order $3$, we have $\dim C_V(T) \ge 2$. If $\dim C_V(T) =3$, then as $\Alt(5)$ is generated by two subgroups of order $3$, we find that an element of order $5$ has fixed points on $V$ and this is impossible. Therefore $\dim C_V(T) = 2$ and $N_Y(T)$ acts upon this subspace.  Let $R \in \Syl_2(N_Y(T))$. Then by Lemma~\ref{Laction-1}(ii), $|R|= 2^4$ and $R\cap Q\cong \Dih(8)$. By (ii) the elements of $R\setminus Q$ have order at least $4$. Since the central involution in $Q$ inverts $V$, we  see that $R$ acts faithfully on $C_V(T)$. It follows that $R$ is isomorphic to a Sylow $2$-subgroup of $\GL_2(3)$ and this proves (iii).
\end{proof}

\begin{table}
\begin{tabular}{|c|c|c|c|c|c|}
\hline Conjugacy Classes&$\Sp_6(2)$&$\Aut(\U_4(2))$&$|C_X(x)|$&$|C_Y(x)|$&$|C_V(x)|$\\

\hline
 $A_1$&2A&2C&$2^9.3^2.5$&$2^5.3^2.5$&$2^4$\\
 $A_2$&2B&2A&$2^9.3^2$&$2^7.3^2$&$2^6$\\
$A_3$&2C&2B&$2^9.3$&$2^6.3$&$2^4$\\
$A_4$&2D&2D&$2^7.3$&$2^5.3$&$2^4$\\
 \hline
\end{tabular}\label{TabInv}
\caption{Involutions in $\Aut(\U_4(2))$ and $\Sp_6(2)$}
\end{table}

The group $\Sp_6(2)$ has a unique $8$-dimensional irreducible module
 over $\GF(2)$ as can be seen for example in \cite{MOAT}. This
module is usually called the \emph{spin module} for $\Sp_6(2)$. On
restriction to any subgroup of  $\Sp_6(2)$ isomorphic
$\Aut(\U_4(2))$ the spin module remains irreducible and is the
unique irreducible module of dimension $8$ over $\GF(2)$ for this
group. In Section 3, we shall refer to this module
as the \emph{natural module} for $\Aut(\U_4(2))$.  The next two lemmas
collect information about the action of certain subgroups and
elements of these two groups on the spin module for $\Sp_6(2)$.

\begin{lemma}\label{SP62facts} Suppose that $X \cong \Sp_6(2)$, $Y$ is a subgroup of $X$ with $Y \cong \Aut(\U_4(2))$  and
$V$ is the $\GF(2)X$-spin module. Then the following hold.
\begin{enumerate}
\item There are exactly four conjugacy classes $A_1,A_2,A_3$ and $A_4$ of involutions in $X$ and, for $1\le i \le 4$,   $A_i \cap Y$ is a conjugacy class of involutions in $Y$.   For each conjugacy class
$A_i$, $1\le i\le 4$, and for $x$ an involution in  $A_i$, Table~1
gives the {\sc Atlas} class name for $A_i$ in both $X$ and $Y$,
$|C_X(x)|$, $|C_Y(x)|$ and $|C_V(x)|$.
\item If $P$ is a parabolic subgroup of shape $2^{5}.\Sp_4(2)$ in $X$,
then $O_2(P)$ contains one involution from $A_1$ and fifteen
involutions from each of $A_2$ and $A_3$. Furthermore, as a
$P/O_2(P)$-module, $O_2(P)$ is an indecomposable extension of the
trivial module by a natural module.
\item If $x \in A_2$, then $\langle x\rangle = Z(C_X(x))$ and
$C_X(x)$ is a maximal subgroup of $X$.
 \item If $f \in X$ has order five, then $C_V(f)=0$.

\item For $v\in V$, $|C_Y(v)|$ and $|C_X(v)|$ are divisible by $3$.

\item For $S \in \syl_2(Y)$,  $|C_V(S)|=|C_{V/C_V(S)}(S)|= 2$.
\item If $S \in \syl_2(X)$ and $x \in N_X(Z(S))$  has order $3$, then $x$ acts fixed-point-freely on $V$.
\item There are no subgroups of  $X$ of order $2^5$ which have all non-trivial
elements in class $A_2$.
\end{enumerate}
\end{lemma}

\begin{proof} The facts in (i) regarding
involutions classes and their centralizers in $X$ and $Y$ are taken
from the {\sc Atlas} \cite[pgs. 26 and 46]{Atlas}--we determine
$|C_V(x)|$ later in the proof. We also  immediately see that
$C_X(x)$ is a maximal subgroup of $X$ for $x\in A_2$. So (iii)
holds.

Let $S \in \syl_2(X)$ and $P_1$, $P_2$ and $P_3$ be the maximal
parabolic subgroups of $X$ containing $S$ with $P_1\sim
2^{5}.\Sp_4(2)$, $P_2\sim 2^6.\SL_3(2)$ and $|P_3|= 2^9.3^2$. Then
the restrictions of  $V$ to $P_i$, $i=1,2,3$ are given in
\cite{PRoh}. In particular, we have that $[V,O_2(P_1)] =
C_V(O_2(P_1))$ has dimension $4$ and, as $P/O_2(P_1)$ modules,
$V/C_V(O_2(P_1)) \cong C_V(O_2(P_1))$ and both are natural
$\Sp_4(2)$-modules. Therefore, the elements of order $5$ in $X$ act
fixed point freely on $V$ which gives (iv).

From the character table of $X$, we read that there are dihedral subgroups of $X$ of order $10$ which contain
involutions from classes $A_1$, $A_3$ and $A_4$. Therefore
$|C_V(x)|= 2^4$ for $x$ in any of these classes. We have that $V$
restricted to a Levi complement $L$ of $P_1$ decomposes as a direct
sum of two natural modules and  so the transvections in $L$
centralize a subspace of dimension $6$ in $V$. These elements are
therefore in class $A_2$. This completes the proof of (i).

 Since $C_V(S)$ is normalized by $P_2$, we
calculate that $Y$ has two orbits on $V^\#$ one of length $135$ and
the other of length $120$. In particular (v) holds.

Since $Z =Z(S)$ contains elements from classes $A_1$, $A_2$ and
$A_3$ which we denote by $z_a$, $z_b$ and $z_c$ respectively,
$N_X(Z)= C_X(Z) \le C_X(z_c) \le P_1\cap P_3 \le C_X(z_a) \cap
C_X(z_b) \le C_X(Z)$. It follows that $N_X(Z) \not \le P_2$ and hence
the elements $d$ of order $3$ in $N_X(Z)$ have $C_V(d)=0$. Thus
(vii) holds.

From Table~1 we have that $Z(S) \le O_2(P_1)$ contains elements from
each of the classes $A_1$, $A_2$ and $A_3$. As $P_1$ centralizes an
element $z$ of $Z(S)$ in class $A_1$ and since $P_1$ acts
transitively on the non-trivial elements of $O_2(P_1)/\langle
z\rangle$. The first part of (ii) holds. The final part of (ii) is
well known and can be, for example, verified by using the Chevalley
commutator formula to calculate that $|[O_2(P),S]|=2^4$ where $S \in
\syl_2(P)$.

Suppose that $B$ is an elementary abelian subgroup of $X$ of order
$2^5$ in which  every involution is in $A_2$. By considering the
restriction of $V$ to $P_1$, we see that $|BO_2(P_1)/O_2(P_1)| \le
2$. Thus $B\cap O_2(P_1)$ contains all the $A_2$-involutions of
$O_2(P_1)$ and is consequently $P_1$ invariant. This contradicts
(ii), so proving part (viii).

We prove (vi). Let $P$ be the parabolic subgroup of $\Aut(\U_4(2))$
of shape $2^4:\Sym(5)$, $R= O_2(P)$ and $S \in \syl_2(P)$. Then as
the elements of order $5$ in $P$ act fixed point freely on $V$,
$C_V(R)=[V,R]$ has dimension $4$. Furthermore, $C_V(R)$ is an
irreducible $P/R$-module and from this we obtain $C_V(S)
=C_{C_V(R)}(S)$ and $C_{C_V(R)/C_V(S)}(S)$ have dimension $1$. Since
$[S,S] \cap R$ has order $2^3$ and $R$ contains only $5$ elements in
class $A_2$, we deduce that $[S,S]$ contains an involution that is
not in class $A_2$. As the preimage of $C_{V/C_V(S)}(S)$ is
centralized by $[S,S]$, we see that $C_{V/C_V(S)}(S)=
C_{C_V(R)/C_V(S)}(S)$ and (vi) follows.
\end{proof}

%
%
%

\begin{lemma}\label{NOTFF}  Suppose that $X \cong \Sp_6(2)$ and $V$ is the
$\GF(2)X$-spin module.  If $F \le X$, $[V,F,F]=0$ and $|V/C_V(F)|\le
|F|$, then there exists $f\in F^\#$ which is not in class $A_2$.
\end{lemma}

\begin{proof}  First of all we note that, as $V$ is
self-dual,  $|[V,F]|=|V/C_V(F)| \le |F|$.

Assume that every non-trivial element of $F$ is in class $A_2$. Then
$2^4\ge |F|
> 2$ by Lemma~\ref{SP62facts} (i) and (viii). If $|F|= 2^2$, then for
$f_1, f_2 \in F^\#$ with $f_1 \ne f_2$ we have $C_V(f_1)= C_V(f_2)=
C_V(F)$. But then $C_V(F)$ is invariant under $\langle
C_X(f_1),C_X(f_2)\rangle = X$ as $C_X(f_1)$ is a maximal subgroup of
$X$ by Lemma~\ref{SP62facts}(iii). Therefore $|V:C_V(F)|\ge 2^3$ and
$|F| \ge 2^3$.

Assume that $P_1$ is a parabolic subgroup of $X$ of shape
$2^{5}.\Sp_4(2)$ such that $F \le P_1$. Set $E = F \cap O_2(P_1)$.
Suppose that $|E| \ge 2^3$. If $|E|= 2^4$, then $E$ contains all the
$A_2$-elements of $O_2(P_1)$ and hence is invariant under the action
of $P_1$.  This contradicts Lemma~\ref{SP62facts}(ii) and so we
conclude that $|E|= 2^3$. Let $P \le P_1$ be the parabolic subgroup
of $P_1$ which normalizes $EZ(P_1)$. Since $E$ contains all the
$A_2$-elements of $EZ(P_1)$, $P$ normalizes $E$. Also, since $P$
normalizes $EZ(P_1)$, $P$ normalizes $Z(S)$ for any $S
\in\syl_2(P)$. Hence $P$
 only normalizes subspaces of even
dimension by Lemma~\ref{SP62facts}(vii). Consequently, as $P$
normalizes $C_V(E)$ and $|C_V(E)| \le 2^5$, we deduce that $C_V(E) =
C_V(O_2(P_1))$ has order $2^4$. Since $E$ acts quadratically on $V$,
$[V,E]= C_V(E)$ and thus $C_V(F)= C_V(E) $. So $|F| = 2^4$ and
hence, as  $|E|= 2^3$, $F \not \le O_2(P_1)$. But then $C_V(F) <
C_V(E)$ which is a contradiction. Hence $|E| \le 2^2$. Because
$O^2(P_1)\setminus O_2(P_1)$ contains no $A_2$-elements, we have
$|F| \le 2^3$  and so $|F|=2^3$. Finally, $[V,F] \ge [V,E] + [V,f]$
for some $f \in F \setminus O_2(P_1)$ and so, as $[V,f] \not \le
[V,O_2(P_1)]$ and $[V,E] \le [V,O_2(P_1)]$ with $|[V,E]|\ge 2^3$, we
have $|[V,F]|
> |F|$, and this is our final contradiction.
\end{proof}

\begin{lemma}\label{88} Suppose that $X \cong \Aut(\U_4(2))$ and $x$
is an involution of $X$ with $C_X(x) \cong 2 \times \Sym(6)$. Let $F
\in \Syl_3(C_X(x))$. If $T \in \Syl_3(X)$ and $F\le T$, then $F \le
J(T)$.
\end{lemma}

\begin{proof}  Note that
$J(T)$ is elementary abelian of order $3^3$. If $Z(T) \le F$,  then
$x \in C_X(Z(T)) \le X'$ by \cite[pg. 26]{Atlas} whereas $x \not\in
X'$. Thus
 $Z(T) \not \le F$. Hence $Z(T)F$ is elementary abelian of order
$3^3$ and so $Z(T)F= J(T)$, and the lemma holds.
\end{proof}

\section{A $2$-local subgroup}

As intimated in Section 1, the raison d'$\hat {\mathrm e}$tre for
Theorem~\ref{closed} is to assist in uncovering the structure of an
involution centralizer in a group satisfying the hypothesis of
Theorem~\ref{Th1}. The main thrust of the proof of
Theorem~\ref{closed} is to show that $Q$ is a strongly closed
$2$-subgroup of $T$ with respect to $G$ where $T \in \syl_2(H)$.
Goldschmidt's classification  of groups with a strongly closed
abelian $2$-subgroup \cite{GO} quickly concludes the proof. We use
the simultaneous notation for conjugacy classes in the groups
$\Sp_6(2)$ and $\Aut(\U_4(2))$ given in Table~1. In the next theorem we use $(3\times \U_4(2)):2$ to indicate the split extension of $3\times \U_4(2)$ by an involution which inverts the normal subgroup of order $3$ and acts as a non-trivial outer automorphism on the normal subgroup isomorphic to $\U_4(2)$.  The case where $H/Q \cong (3\times \U_4(2)):2$ does not arise in this paper; however it will find application in work in preparation by  Parker and Stroth which characterizes automorphism groups related to $\mathrm {PSU}_6(2)$.

\begin{theorem}\label{closed} Suppose that $G$ is a   finite group, $Q$ is a
subgroup of $G$ and $H= N_G(Q)$. Assume that the following  hold
\begin{enumerate}
\item $H/Q \cong \Aut(\U_4(2))$, $(3\times \U_4(2)):2$ or $\Sp_6(2)$;
\item $Q=C_G(Q)$ is a minimal
normal subgroup of $H$ and is elementary abelian of order $2^8$;
\item $H$ controls fusion of elements of $H$ of order $3$; and
\item if $g \in G\setminus H$ and $d \in H\cap H^g$ has order $3$,
then $C_Q(d)=1$.
\end{enumerate}
Then $G= HO_{2'}(G)$.
\end{theorem}

\begin{proof} Let $T \in \Syl_2(H)$.
To begin with we note that as a $\GF(2)H$-module, $Q$ is isomorphic
to the $\Sp_6(2)$ spin-module when $H/Q\cong \Sp_6(2)$ and to the natural $\Aut(\U_4(2))$-module when $H/Q\cong \Aut(\U_4(2))$. If $H/Q \cong (3\times U_4(2)):2$, then letting $H_0$ be the subgroup of index $3$ in $H$, $Q$ is isomorphic to  the natural $H_0/Q$-module.

\begin{claim}\label{3'group} Suppose that $g \in G$ and $y \in (Q^g \cap H)\setminus Q$. Then $C_H(y)$ is a $3'$-group.
\end{claim}

Let $y \in(Q^g \cap H)\setminus Q$ and suppose that $3$ divides
$|C_H(y)|$, $S \in \syl_3(C_H(y))$ and $x= y^{g^{-1}}$. Then $x \in
Q$ and $|C_H(x)|$ is divisible by $3$ by Lemma~\ref{SP62facts} (v).
Let $P \in \Syl_3(C_H(x))$. If $P \not \in \syl_3(C_G(x))$, then
$N_{C_G(x)}(P) \not \le H$ and so there exists $n \in
N_{C_G(x)}(P)\setminus H$ such that $P \le H\cap H^n$. Since, for $d
\in P$ of order $3$, $x\in C_Q(d)$,  this contradicts assumption
(iv). Hence $P \in \syl_3(C_G(x))$ and therefore $P^g \in
\syl_3(C_G(y))$. Since $S$ is a $3$-subgroup of $C_G(y)$, there is
an $k \in C_G(y)$ such that $P^{gk} \ge S$. By assumption (iii), $H$
controls fusion of elements of order $3$ in $H$. Hence, as each
element of $S$ is $G$-conjugate to an element of $P$, each element
of $S$ is $H$-conjugate to an element of $P$. Now, as $ x\in C_Q(P)$
and $Q$ is normal in $H$, for all elements of $s\in S$ we have $C_Q(s)
\neq 1$. Since $S \le H\cap H^{gk}$, we then get $gk \in H$ by (iv).
Thus $y=x^{gk} \in Q^{gk}=Q$ and we have a contradiction as $y \not
\in Q$. Therefore, $3$ does not divide $|C_H(y)|$ as claimed.\qedc

\begin{claim}\label{ya2} Let $g \in G$ and suppose  $y \in (Q^g \cap H)\setminus
Q$. Then $yQ$ is an $A_2$-involution  in $H/Q$ and $C_H(y)Q\in
\syl_2(H)$. Furthermore, $H/Q\not\cong (3\times \U_4(2)):2$.
\end{claim}

If $yQ$ is not in the $A_2$-class of $H/Q$, then, by
Lemma~\ref{SP62facts}(i), $C_Q(y)=[Q,y]$ and so Lemma~\ref{orbits}
gives  $C_H(y)Q/Q= C_{H/Q}(y)$. Thus $C_H(y)$ is not a $3'$-group by
Lemma~\ref{SP62facts}(i) again, and this is contrary to
\ref{3'group}.
 Hence $yQ$ is in the  $A_2$-class of $H/Q$.  Let $D$ be the full preimage of $C_{H/Q}(yQ)$ in $H$. Then $D$ operates on the set  $\mathcal I$ of involutions contained in $Qy$. From Lemma~\ref{SP62facts}(i),  $|C_Q(y)|= 2^6$ and $|D/Q|$ is divisible by $9$. In particular, $|\mathcal I|= 64$.  By \ref{3'group}, $|D:C_H(y)|$ is divisible by $9$ and, by Lemma~\ref{SP62facts}(i), $|Q:C_Q(y)|=2^2$. Therefore $|D:C_H(y)|$ is divisible by $36$. Since $D$ obviously cannot have an orbit of length $72$ on a set of 64 elements, we conclude that $|D:C_H(y)|=36$.  If $H/Q\cong (3\times \U_4(2)):2$, then in fact $27$ divides $|D|$ and we conclude that $|C_H(y)|$ is divisible by $3$, contrary to \ref{3'group}. Thus  $H/Q\not\cong (3\times \U_4(2)):2$. If $H/Q \cong \Sp_6(2)$, we get $|C_H(y)|= 2^{15}$  and, if $H/Q \cong \Aut(\U_4(2))$, we get $2^{13}$. Therefore, as $|Q:C_Q(y)|= 2^2$,
$C_H(y)Q \in \syl_2(H)$. So \ref{ya2} holds.\qedc

We note  that \ref{ya2} applies equally well to show that
involutions in $(Q \cap H^g)Q^g/Q^g$ are in the $A_2$-class of
$H^g/Q^g$.

\begin{claim}\label{weakclose} $Q$ is weakly closed in $H$ with respect to $G$. In particular, $T\in
\syl_2(G)$.
\end{claim}

Suppose that \ref{weakclose} is false.  Then, by Lemma~\ref{wcl1},
there  exists $g\in G\setminus H$ such that $Q^g$ and $Q$ normalize
each other. In particular, $Q^g \le H$. Hence we may assume that $|Q:C_Q(Q^g)|\le |Q^gQ/Q|$.  By
\ref{ya2} the non-trivial elements of $Q^g Q/Q$ are all in $H/Q$
class $A_2$. These two facts together contradict Lemma~\ref{NOTFF}.
Therefore $Q$ is weakly closed in $H$ with respect to $G$ and consequently
$\syl_2(H) \subseteq \syl_2(G)$. \qedc

Aiming for a contradiction we now suppose that $Q$ is not  strongly
closed in $T$ with respect to $G$.

\begin{claim}\label{norm} We can select $g\in G$ and $y \in (Q^g \cap H)\setminus Q$
so that $C_H(y) \le H^g$.\end{claim}

Since $Q$ is not strongly closed in $T$ ($\le H$), there exists
$g\in G$ and $y \in (Q^g \cap H)\setminus Q$. Clearly $Q^g \le
C_G(y)$, and so we may select a Sylow $2$-subgroup $T_1$ of $C_G(y)$
such that $T_1$ contains $Q^g$. Since $C_H(y)$ is a $2$-group by
\ref{ya2}, there exists a Sylow $2$-subgroup $T_2$ of $C_G(y)$ which
contains $C_H(y)$. Thus there is an $f \in C_G(y)$ such that $T_1^f=
T_2$. Because $Q$ is weakly closed in $H$ and $Q^{gf} \le T_2$,
$C_H(y) \le T_2 \le N_{G}(Q^{gf})=H^{gf}$. Since $f \in C_G(y)$, $y
\in (Q^{gf}\cap H) \setminus Q$. Thus we may replace $g$ by $gf$ and
we have proved \ref{norm}.\qedc

Choosing $g$ and $y$ as in \ref{norm}, we set $W=C_H(y)Q^g$.

\begin{claim}\label{315} There exists a Sylow $2$-subgroup $T_0$ of $H^g$ which
normalizes $Q \cap Q^g$ and contains $W$ . Furthermore, $|T_0:W| \le
2$.
\end{claim}

Since $C_H(y) Q \in \syl_2(H)$ by \ref{ya2}, and  $C_H(y) Q$
normalizes $Q\cap Q^g$ by \ref{norm}, $N_H(Q \cap Q^g)$ contains a
Sylow $2$-subgroup of $G$ by \ref{weakclose}. Since $W$ normalizes
$Q \cap Q^g$, there is a $T_0 \in \syl_2(N_G(Q\cap Q^g))$ with $T_0
\ge W$. Therefore, as $Q^g$ is weakly closed in $W$, $T_0 \le H^g$.
Since $|Q:C_Q(y)|=4$, we have $|T_0:W| \le 4$ by \ref{ya2}. If
$|T_0:W| =4$, then  we must have $Q^g \le C_H(y)$ which contradicts
$Q$ being weakly closed in $H$ and $Q \not=Q^g$. Hence $|T_0:W| \le
2$. \qedc

Let $Z_2(T_0)$ be the second centre of $T_0$ where $T_0$ is as in
\ref{315}. Then, as $|Z_2(T_0)| =4$ by Lemma~\ref{SP62facts}(vi) and
$Q \cap Q^g$ is normal in $T_0$, we either have $|Q \cap Q^g|\le 2$,
or $Z_2(T_0) \le Q\cap Q^g$. Since $|T_0:W|  \le 2$, $ C_{Q^g}(W)
\le Z_2(T_0)$. From $y \in C_{Q^g}(W) \le Z_2(T_0)$ and $y \not\in
Q$, we must have $|Q\cap Q^g|\le 2$. Since $yQ$ is in $H/Q$ class
$A_2$, we have $|C_Q(y)|= 2^6$. Hence $|C_Q(y)Q^g/Q^g|= |C_{Q}(y):Q
\cap Q^g| \ge 2^5$ and, by \ref{ya2}, all the involutions of
$C_{Q}(y)Q^g/Q^g$ are in $H^g/Q^g$ class $A_2$, which contradicts
Lemma~\ref{SP62facts} (viii). We have therefore shown that $Q$ is
strongly closed in $T$ with respect to $G$.

Set $M=\langle Q^G\rangle$. If $M \not = QO_{2'}(G)$, then $|M:Q|$
is even and hence we have $T \cap M > Q$ by \ref{weakclose}. But
then $\langle (T\cap M)^H\rangle$ has index at most $2$ in $H$ and
is contained in $M$. Finally, applying Goldschmidt's Theorem
\cite{GO}, we see that the possible composition factors of
$M/O_{2',2}(M)$ do not involve either $\U_4(2)$ or $\Sp_6(2)$. Thus
$M=QO_{2'}(G)$ and the Frattini Argument completes the proof of the
theorem.
\end{proof}

\section{Part of the $3$-local structure}

Having now gathered together our prerequisite results, we  are ready
to begin the proof of Theorem~\ref{Th1}. Thus for the remainder of
this article we assume that $G$ is a finite group with $S$ a Sylow
$3$-subgroup of $G$ and $Z= Z(S)$. Additionally, we assume that $Z$
is not weakly closed in $S$ with respect to $G$ and $C_G(Z)$ has
shape $3^{1+4}.2^{1+4}.\Alt(5)$ as described in the hypothesis of
Theorem~\ref{Th1}. We set $L= N_G(Z)$, $L_* = C_G(Z)$, $Q= O_{3}(L)$
and let $P\in \syl_2(O_{3,2}(L_*))$. So $P$ and $Q$ are extraspecial
of order $2^5$ and $3^5$ respectively and $O_{3,2}(L_*) =PQ$.
Furthermore, $O_3(L_*)=Q$. Let $\langle u \rangle = Z(P)$.

We begin by fleshing out the structure and embeddings of these
groups. In the next proof we use the fact that $\Sp_4(3)$ contains
no subgroup isomorphic to $\Alt(5)$. This is easy to see as the
$2$-rank of both $\Sp_4(3)$ and  $\Alt(5)$ is $2$ whereas $\Alt(5)$
has no non-trivial central elements.

\begin{lemma}\label{lem1}
\begin{enumerate}
\item $Z=Z(Q)$ has order $3$.
\item $L_*$ and $L$ are $3$-constrained.
\item  $L_*/Q$ is $2$-constrained,  acts irreducibly on $Q/Z$ and $P \cong 2^{1+4}_-$.
\item $Q$ is extraspecial of $+$-type.
\end{enumerate}
\end{lemma}

\begin{proof} Since $Z$ is normal in $L^*$, $Z \le O_3(L^*)=Q$ and so, as $Q$ is extraspecial,
$Z=Z(Q)$ has order $3$. This is (i).

Suppose  that $C_L(Q) \not \le Q$. Then $C_L(Q)Q/Q$ is a non-trivial
normal subgroup of $L_*/Q$. Let $D \in \syl_3(C_L(Q))$. Then $|D|\le
9$ and hence is abelian. If $D >Z$, then $DQ= S$ and hence $D  \le
Z(S)=Z$ which is a contradiction. Thus $D = Z \le Q$ by (i). The
assumed structure of $L_*$ now indicates that $C_L(Q) \le QP$. In
particular,  $L_*/C_L(Q)$ has a composition factor isomorphic to
$\Alt(5)$. As $Q$ is extraspecial, the commutator map defines a
symplectic form on $Q/Z$ and so $\Out(Q)$ is isomorphic to a
subgroup of $\GSp_4(3)$. Since $\Sp_4(3)$ has no subgroups
isomorphic to $\Alt(5)$, $C_L(Q) < QP$. If $C_{L}(Q)Q = \langle
u\rangle Q$, then $PC_{L}(Q)Q/Q$ has $2$-rank $4$, contrary to the
$2$-rank of $\Sp_4(3)$ being $2$. Thus $\langle u\rangle Q /Q<
C_{L}(Q)Q/Q < PQ/Q$. In this case, $C_{L_*/Q}(PQ/Q)$ must contain a
component $L_1$ isomorphic to $\Alt(5)$ or $\SL_2(5)$. The former
case being impossible, we get $L_1\cong \SL_2(5)$. Since $L_1 \cap
PQ/Q$ is normal of order $2$ we deduce that $L_1 \ge \langle
u\rangle Q/Q$, and once again we have $L_1C_{L}(Q)Q/Q \cong \Alt(5)$
which is our final contradiction. Hence $C_L(Q)=Z$ and (ii) holds.

Part (iii) follows from  Lemma~\ref{Laction}, since $L_*/Q$ acts
faithfully on $Q/Z$ and $PQ/Q$ is extraspecial.

Finally (iv) is a consequence of (iii) and \cite[Lemma
2.8]{PRSymplectic}.
\end{proof}

\begin{lemma}\label{inv1} Suppose that $s$ is an involution of $L_*$ with $sQ\not= uQ$.
Then the following hold. \begin{enumerate}
\item $s \in PQ$.
\item $C_{L_*}(s)PQ/PQ \cong \Alt(4)$.
\item  $Q= C_Q(s)[Q,s]$, $[C_Q(s),[Q,s]]=1$ and $C_Q(s) \cong
[Q,s] \cong 3^{1+2}_+$.
\item $C_{PQ}(s) \sim 3^{1+2}_+.(\Q(8)\times 2)$ and $O_{3'}(C_{PQ}(s)) =\langle s \rangle$.
\item $C_{L^*}(u)/O_2(C_{L^*}(u))\cong 3\times \Alt(5)$.
\end{enumerate}
\end{lemma}

\begin{proof} Part (i) follows from Lemma~\ref{Laction2}(ii) and
part (ii) comes from Lemma~\ref{Laction-1} (i).

 Because  $Q=
C_Q(s)[Q,s]$   the Three Subgroup Lemma shows that
$[C_Q(s),[Q,s]]=1$. Thus, as $sQ\not = uQ$, $[Q,s] < Q$ and so, as
$s$ does not centralize $Q$, we deduce that  $C_Q(s) \cong [Q,s]
\cong 3^{1+2}_+$ from Lemma~\ref{lem1}(iv).
Part (iv) follows from Lemma~\ref{Laction-1} (i) and part (iii).
Since $C_{L^*}(Q) = Z$, and $L^*$ acts irreducibly on $Q/Z$, $u$ inverts $Q/Z$. Therefore, $C_Q(u)= Z$ and by the Frattini Argument, $C_{L^*}(u)Q= L^*$.
  Now $C_{L_*}(u)/O_2(C_{L_*}(u))$ has shape
$3.\Alt(5)$ and hence is isomorphic to $3\times \Alt(5)$ as the
Schur multiplier of $\Alt(5)$ has order $2$.  Thus (v) holds.

\end{proof}

The next lemma shines a light on the structure of $C_{L^*}(s)$ for $s \in L^*$ an involution with $sQ\neq uQ$.

\begin{lemma}\label{inv2} Suppose that $s $ is an involution of $L_*$ with $sQ \not= uQ$.
Then the following hold.
\begin{enumerate}
\item $[O_2(C_{L_*}(s) ),O_3(C_{L_*}(s))]=1$.
\item $O_2(C_{L_*}(s) )=O_{3'}(C_{L_*}(s)) \cong
\Q(8)$.  \item $C_{L_*}(s)/O_2(C_{L_*}(s) )\sim 3^{1+2}_+.\SL_2(3)$
is isomorphic to the centralizer of a non-trivial $3$-central
element in $\PSp_4(3)$.
\item If $b \in C_{L_*}(s)$ has order $3$ and $b \not \in Q$, then
$C_{O_{3'}(C_{L_*}(s))}(b)= \langle s \rangle$.
\end{enumerate}
\end{lemma}

\begin{proof} Part (i) is trivial (and is included as it illuminates the structure of $C_{L_*}(s)$).
Set $Y = C_{L_*}(s)$, $W= C_Q(s)= Q\cap Y$ and select an involution
of $Qu$ which centralizes $s$ and, for convenience, call it $u$.
Then, by Lemma~\ref{inv1} (iv), $W \cong 3^{1+2}_+$. Therefore
$Y/C_{Y}(W)$ embeds into $\Aut(3^{1+2}_+) \sim 3^2.\GL_2(3)$. As $W$
is extraspecial, $WC_{Y}(W)/C_{Y}(W) \cong 3^2$. Let $X = C_{Y}(W)$.
Since $(QP\cap Y)Q/Q \cong \Q(8) \times 2$ by Lemma~\ref{inv1} (iv)
and since $u$ inverts $W/Z$, $C_{QP\cap Y}(W) = C_W(W)\langle
s\rangle = Z\langle s\rangle$. Hence, as $X$ is normal in $Y$, we
have
$$[X,C_{QP}(s)]\le X \cap C_{QP}(s) = Z\langle s\rangle.$$
As the elements of order $3$ in $Y\setminus W$ act non-trivially on
$(PQ\cap Y)Q/Q$, we get $X \le C_{FQ}(s)$ where $F \in \syl_2(Y)$.
Additionally, as $Y/Q$ is $2$-closed, we have $Y/C_{Y}(W) \sim
3^2.\SL_2(3)$ and $C_{Y}(W)$ has order $2^3.3$. It follows that
$|O_2(Y)|=2^3$. Noting that $O_2(Y)$ and $u$ are in a common Sylow
$2$-subgroup of $Y$, $[Q,s] = C_{Q}(su)$ and that $O_2(Y)$ acts
faithfully on $[Q,s]$ by the $3$-constraint of $L_*$. By applying
the above conclusions to the involution $su$, we obtain $O_2(Y) \cong
\Q(8)$. As $O_2(Y)= O_{3'}(Y)$,  (ii) holds.

Now we  have $Y/O_2(Y) \sim 3^{1+2}_+.\SL_2(3)$ and
$O_2(Y/O_2(Y))=1$.  From Lemma~\ref{inv1} (v), $C_{L_*}(u)$ has
elementary abelian Sylow $3$-subgroups. It follows that the Sylow
$3$-subgroups of $C_{Y/O_2(Y)}(uO_2(Y))$ are  elementary abelian.
 So, using Lemma~\ref{psp43new} the conclusion in
(iii)  holds.

Since by Lemma~\ref{inv1} (ii), $C_{L^*}(s)PQ/PQ\cong \Alt(4)$ and $b \not\in Q$, we have $C_{O_{3'}(C_{L^*}(s))}(b) \le PQ$. Thus (iv) follows from Lemma~\ref{inv1}(v).

\end{proof}

Another, less precise, way of recording Lemma~\ref{inv2} is to say
that $C_{L_*}(s)$ has shape $(3^{1+2}_+\times \Q(8)).\SL_2(3)$.

\begin{lemma}\label{SACTION}
 $C_{Q/Z}(S) = [Q/Z,S]$ has order $3^2$ and $[Q,S]$ is
elementary abelian of order $3^3$. In particular $C_Q([Q,S])=[Q,S]$.
\end{lemma}

\begin{proof}

 Since $L_*/Q \sim 2^{1+4}_-.\Alt(5)$, Lemma~\ref{inv1} (ii)
 implies that there is an involution $s \in PQ$
which centralizes $S/Q$ and satisfies $sQ\neq uQ$. Hence  $S$ normalizes $Q_1=C_Q(s)$ and
$Q_2=[Q,s]$. Thus, by Lemma~\ref{inv1}(iii), $C_{Q/Z}(S)  =
C_{Q_1/Z}(S)C_{Q_2/Z}(S)$. By  Lemma~\ref{Laction-1} (ii),  $sQ$ and
$suQ$ are conjugate in $N_{PQ/Q}(S/Q) \cong \Dih(8)$ by an element
$fQ$ say. Since $u$  inverts $Q/Z$ by Lemma~\ref{lem1}(iii), we get
that $Q_1^f= Q_2$. Thus  $|C_{Q_1/Z}(S)| = |C_{Q_2/Z}(S)|$.
Therefore, as $L_*$ is $3$-constrained by Lemma~\ref{lem1} (ii),
$|C_{Q/Z}(S)|= 3^2$.  Since, for $i=1,2$,  $[Q_i/Z,S] \le
C_{Q_i/Z}(S)$, we get that $C_{Q/Z}(S)= [Q/Z,S]$ has order $3^2$ as
claimed. The Three Subgroup Lemma and $Q$ being of exponent $3$
shows that $[Q,S]$ is  elementary abelian. Finally, noting that $Z
\le [Q,S]$  we have $|[Q,S]| =3^3$. In particular, $[Q,S]$ is a maximal abelian subgroup of $Q$.
\end{proof}

We now put  $J = C_S([Q,S])$, and start the investigation of the
$3$-local subgroup  $M=N_G(J)$. Set $M_* = O^{3'}(M)$.

\begin{lemma}\label{New1}
The following hold.
\begin{enumerate}
\item $J= J(S)$ is elementary abelian of order $3^4$;
\item $S= JQ$;
\item no element of $S\setminus J$ is acts quadratically on $J$; and
\item every element of order $3$ in $S$ is contained in $J \cup Q$.
\end{enumerate}
\end{lemma}

\begin{proof}
From Lemma~\ref{SACTION}, we have that $C_Q([Q,S])=[Q,S]$. It follows that $|J|\le 3^4$. Let $b \in C_S(u)\setminus Q$. Then, by Lemma~\ref{inv1}(v), $b$ has order $3$. Also, as $[Q,S]$ is abelian and $u$ inverts $Q/Z$ and centralizes $Z$, we have  $[[Q,S],u] \cap Z=1$ and $[Q,S]=[[Q,S],u]Z$. Since $[[[Q,S],u],b]\le [[Q,S],u] \cap [Q,S,S] =[[Q,S],u] \cap Z$, we see that $b$ centralizes $[Q,S]$ and conclude that $J$ is elementary abelian of order $3^4$. Suppose that $A$ is an abelian subgroup of $S$ of order at least $3^4$. Then $3^3\ge |A\cap Q|\ge 3^3$. Therefore $A\cap Q$ has order $3^3$ and $[S,A\cap Q]=[AQ,A\cap Q]\le [Q,Q]= Z$. Hence $A \cap Q = [Q,S]$ by Lemma~\ref{SACTION}. But then $A \le J$ and we have $A=J$. Thus $J= J(S)$.

Since $J \not \le Q$, (ii) is obvious.

We get  $N_{L^*}(S)/S \cong \SDih(16)$ from Lemma~\ref{Laction2}(iii), and so $N_{L^*}(S)$ acts transitively on the elements of $S/J$. Thus if any element of $S/J$ acts quadratically on $J$, then they all do. So suppose that $s \in PQ$ with $sU \neq uQ$, $[J,s] \le Q$ and $x \in C_Q(s)J/J$ is non-trivial and acts quadratically on $J$. Then $1=[J,x,x]= [J,C_Q(s),C_Q(s)]$. In particular, $[J,x] \le Z(C_Q(s))$. By Lemma~\ref{inv1}(iii), $C_Q(s)$ is extraspecial, and hence $[J,C_Q(s)]\le  Z(C_Q(s))= Z$. Now using Lemma~\ref{SACTION}, we have $C_Q(s) \le [Q,S]$. Since the former group is extraspecial and the latter group is abelian, we have a contradiction. This proves (iii)

For (iv) assume for a contradiction that $x \in S= JQ$ has order $3$ and that $x \not \in J\cup Q$. Then $x = jq$ where $j \in J\setminus Q$ and $q \in Q\setminus J$. Since $x^3=1$ and both $Q$ and $J$ have exponent $3$, we have $jqj=q^2j^2q^2$ and $qjqj= j^2q^2$. Hence, using the fact that $J$ is a normal abelian subgroup of $S$, we get
\begin{eqnarray*}
[j,q]^q=q^2j^2q^2jq^2=jqj^2q^2= qj^2q^2j=q^2jqj^2= j^2q^2jq=[j,q].
\end{eqnarray*}
Since $[[Q,S],q] = Z \le C_J(q)$ and $J= \langle j\rangle[Q,S]$, we now have $[J,q]\le C_J(q)$ and this contradicts (iii). Hence (iv) holds.
\end{proof}

\begin{lemma}\label{New2} The following hold.
\begin{enumerate}
\item If $X \le Q$  has order $3$ with $Z \neq X$, then $X \not \le C_L(ZX)'$.
\item $Z$ is weakly closed in $Q$. \item If $g \in G$ and $Z^g \le S$, then $Z^g \le J$.\end{enumerate}
\end{lemma}

\begin{proof}
Suppose $X \le Q$ has order $3$ and $Z \neq X$. Then, by Lemma~\ref{Laction2}(i), we may assume that $ZX$ is normal in $S$. Let $T = C_S(ZX)$. Using Lemma~\ref{Laction2}(i) again we get $[C_Q(ZX),O_{3,2}(C_L(ZX))]$ is extraspecial of order $27$ and so $X \not \le [C_Q(ZX),O_{3,2}(C_L(ZX))]$. It follows that $X \not \le  C_L(ZX)'$. This is (i).

 Suppose that $X\in Z^G\setminus Z$ and $X \le Q$.  Let $T = C_G(ZX)$. Then $Z \le T'$ and $X \not \le T'$ by (i). On the other hand $T \le C_G(X)$ and so, as $Z$ and $X$ are $G$-conjugate, it follows that $Z \le O_3(C_G(X))$. But then, the situation is symmetric and so $X \le T'$
and this is a contradiction. Hence (ii) holds.

The third statement follows from part (ii) and Lemma~\ref{New1}(iv).
\end{proof}

\begin{lemma}\label{centralizers} The following hold:
\begin{enumerate}
\item $L\cap M = N_G(S)$.
\item $C_G(J)= C_G([Q,S]) = J$.
\end{enumerate}
\end{lemma}

\begin{proof}  We have $N_G(S)$ normalizes
$Z(S)= Z(Q)$ and $J= J(S)$. Hence $N_G(S) \le L\cap M$. Since $L\cap
M$ normalizes $S=QJ$, (i) holds.

From  $Z \le [Q,S]$, $C_G([Q,S]) \le L_*$ and also, by
Lemma~\ref{extraauto}, $C_G([Q,S])$ is a $3$-group. Since
$C_G([Q,S]) \cap Q= [Q,S]$, we have $|C_G([Q,S])| \le 3^4$ and hence
$C_G([Q,S])= J$ as claimed in (iii).
\end{proof}

\begin{lemma}\label{LMStruct}\begin{enumerate}
\item There are exactly ten $G$-conjugates of $Z$ in $J$.
\item $|L/L_*|=2$, $L \sim 3^{1+4}_+.2^{1+4}_-.\Sym(5)$.
\item $M/J \cong \CO_4^-(3)$ the group of all similitudes of a non-degenerate quadratic form of $-$-type, $M_*/J \cong \Omega_4^-(3)$ and $M/M^* \cong \Dih(8)$.
\end{enumerate}
\end{lemma}

\begin{proof} Since $Z$ is not weakly closed in $S$,
Lemma~\ref{New2} (ii) and (iii) imply that there exists $g \in G$
such that $X= Z^g \le J$ and $X \neq Z$. Since $J$ is abelian, $J$
centralizes $ZX$ and, by Lemma~\ref{SACTION},
$N_S(ZX)= [Q,S]X= J$. Thus there are nine $S$-conjugates of $X$ in
$J$. This shows that the number of $G$-conjugates of $Z$ in $J$ is
congruent to  $1$ modulo $9$. Since, by Lemmas~\ref{fusion} and
\ref{New1} (i), $M$ controls $G$-fusion in $J$, all the
$G$-conjugates of $Z$ in $J$ are conjugate in $M$. Because there is
a unique conjugate of $Z$ in $J \cap Q$ by Lemma~\ref{New2}(ii), we
deduce that $|Z^M| \le 28$. Since $M/J$ acts faithfully on $J$ by
Lemma~\ref{centralizers} (i), we have that $M/J$ is isomorphic to a
subgroup of $\GL_4(3)$. Now $|\GL_4(3)|$ is not divisible by either
7 or 19 and so there is no choice other than $|Z^M|=10$. Hence (i)
holds.

Since $J$ is characteristic in $S$, $N_{L_*}(S) \le M$. Thus, as
$X^S = Z^M \setminus \{Z\}$ and $N_{L_*}(S)$ normalizes $Z$,
$N_{N_{L_*}(S)}(X)S= N_{L_*}(S)$. In particular, $X$ is normalized
by a Sylow $2$-subgroup $T$ of $N_{L_*}(S)$. Since $XQP/QP$ is
inverted in $L_*/QP$, we must have that $X$ is inverted by an
element in $T$. Hence $L > L_*$ and now (ii) follows from
Lemma~\ref{Laction}.

 From (ii) we have $|N_{L}(S)/J| = 2^5.3^2$. Therefore
$|M/J|= 2^6.3^2.5$ by (i). By Lemmas~\ref{GO4minus} and ~\ref{New1}, we have that $Z^M$ is the set of singular $1$-spaces of a quadratic form of $-$-type and by Lemma~\ref{quadratic form} we have that $M$ is isomorphic to a subgroup of $\CO_4^-(3)$. Since the latter group has order $2^6.3^2.5$, this proves (iii).
\end{proof}

Define $M_0 = M_*O_{3,2}(M)$ and let $t \in N_P(S)$ be an involution
with $t \neq u$. Finally set $M_1 = \langle t \rangle M_0$ and
$B=[J,t]$. Note that $M_0/J \cong 2\times \Omega _4^-(3) = \SO_4^-(3)$.

\begin{lemma}\label{CStnew} $C_S(t) \in\syl_3(C_{L_*}(t))$.
\end{lemma}

\begin{proof} This follows directly from Lemma~\ref{inv2}(iii).
\end{proof}

\begin{lemma}\label{JB} $B \le Q$, $|B|=3$ and $|C_J(t)|= 3^3$. In particular, $t$ acts as a reflection on the quadratic space $J$.
\end{lemma}

\begin{proof} From the choice of $t$, we have that $B=[J,t] \le J\cap
Q= [Q,S]$. Since $t \in L_*$, $t$ centralizes $Z$ and so $|B| \le
9$. If $|B|=9$, then $[Q,S]=[[Q,S],u]Z = BZ$ and so $[[Q,S],ut] \le
Z$. Since $ut$ centralizes $Z$ and $J/[Q,S]$, we reason that $ut$
centralizes $J$ and, because $C_J(J)= J$ by Lemma~\ref{centralizers}
(ii), this means that $u=t$ contrary to the choice of $t$. Thus
$|B|=3$ and $|C_J(t)|= 3^3$ as claimed. In particular, $t$ acts as a reflection on $J$.
\end{proof}

\begin{lemma}\label{M0} $M_1/J \cong \GO_4^-(3) \cong 2\times
\Sym(6)$.
\end{lemma}

\begin{proof} Since $t$ acts as a reflection, this is clear.
\end{proof}

\begin{lemma}\label{JOrbs} $M$ has two orbits on  the subgroups of $J$ of order $3$.
One  is $Z^M$ and has length $10$ and the other is $B^M$  and has
length $30$. Furthermore, $N_M(Z)/J \sim (2 \times 3^2).\SD(16)$ and
$N_{M}(B)/J \cong 2 \times 2 \times \Sym(4)$.
\end{lemma}

\begin{proof} We have seen in Lemmas~\ref{centralizers} (i) and \ref{LMStruct} that $|Z^M|=10$
and $N_{M}(Z)= N_G(Z)= L\cap M$. The structure of $N_M(Z)/J$ can be
extracted from Lemma~\ref{LMStruct} (ii).

Suppose that $X$ is a subgroup of $J$ of order $3$ which is not in
$Z^M$. Then $X$ is not $3$-central and therefore corresponds to a non-singular subspace. Since $\CO_4^-(3)$ is transitive on such subgroups, we have that $|X^M|= 30$, as claimed.
Furthermore, in $\CO_4^-(3)$, the subgroup of index $30$ is
contained in $\GO_4^-(3)$. Thus Lemma~\ref{M0} implies that $N_{M}(X)/J \cong 2 \times \GO_3(3)
\cong 2 \times 2 \times \Sym(4)$. Finally we note that $B \le J\cap
Q$ and $B \neq Z$ and so $B^M = X^M$ by Lemma~\ref{New2}(ii).
\end{proof}

%
%
%

\section{The centralizer of $B$}

In this brief section we uncover the structure of $C_G(B)$. We
maintain the notation of the previous section. So $t \in N_P(S)$ is
an involution with $t \neq u$ and $B= [J,t]$.

\begin{lemma}\label{JSig1}  $\signal_{L_*}(J,3')=\{1\}$.
\end{lemma}

\begin{proof} Suppose that $R \in\signal_{L_*}(J,3')$. Then, as $R$ is normalized by $J$ and normalizes $Q$, $R$ centralizes $Q
\cap J=[Q,S]$. Hence $R \le J$ by Lemma~\ref{centralizers} (ii) and
so $R=1$.
\end{proof}

We now extend the scope of the last lemma to the whole of $G$.

\begin{lemma}\label{JSig2} $\signal_{G}(J,3')=\{1\}$.
\end{lemma}

\begin{proof} Suppose that $R\in\signal_{G}(J,3')$.
Then $R = \langle C_R(H)\mid |J:H|=3\rangle$. By
Lemmas~\ref{hyperplanes} and \ref{JOrbs}, each $H$ with
$|J:H|=3$ contains a $M$-conjugate of $Z$. Thus $$R = \langle
C_R(Y)\mid Y \le J \mbox{ and } Y \mbox{ is } M\mbox{-conjugate of }
Z\rangle.$$ Since, for each  $Y\in Z^M$,
$C_R(Y)\in\signal_{C_G(Y)}(J,3')$, Lemma~\ref{JSig1} implies that
$C_R(Y)=1$. Thus $R=1$ and the lemma holds.
\end{proof}

\begin{lemma}\label{Prince1}  We have that $C_{L_*}(B)/B$ is isomorphic to the centralizer of a
non-trivial $3$-central element in $\PSp_4(3)$. Furthermore,
$C_L(B)/B$ inverts $ZB/B$.
\end{lemma}
\begin{proof} Since $Q$ is extraspecial of exponent $3$, we have $C_{Q} (B) \cong 3
\times 3^{1+2}_+$. From Lemma~\ref{Laction2} (i), we have that
$C_{L_*}(B)Q/Q \cong \SL_2(3)$. Thus $C_{L_*}(B)/B\sim
3^{1+2}_+.\SL_2(3)$. Let $U =O_2(C_{L_*}(B))$. Then $Q\ge
C_Q(U) \ge C_Q(B)$. Thus $|Q:C_Q(U)|\le 3$. Since in $\Sp_4(3)$ the subgroup centralizing a hyperplane of the natural $\GF(3)\Sp_4(3)$-module has order $3$, we get
 $U=1$. Since $J \le C_{L_*}(B)$ and $J/B$ is elementary
abelian of order $3^3$, $C_{L_*}(B)/B$ satisfies the hypothesis of
Lemma~\ref{cen3psp43} and so $C_{L_*}(B)/B$  is isomorphic to the
centralizer of a non-trivial $3$-central element in $\PSp_4(3)$.

By Lemma~\ref{Laction2} (i), $L_*$ acts transitively on $(Q/Z)^\#$
and so, as $Q$ is extraspecial, $L_*$ acts transitively on
$Q\setminus Z$. Consequently $C_{L}(B) > C_{L_*}(B)$ and so $ZB/B$
is inverted by $C_{L}(B)$.
\end{proof}

\begin{lemma}\label{Prince2} We have $C_G(B)\cong 3 \times \Aut(\U_4(2))$, $N_G(B) \cong \Sym(3) \times \Aut(\U_4(2))$ and $t$ centralizes $O^3(C_G(B))$.
\end{lemma}

\begin{proof}

Lemmas~\ref{JSig2} and \ref{Prince1} imply that $C_G(B)/B$ satisfies
the hypotheses of Theorem~\ref{PrinceThm}. Furthermore by
Lemma~\ref{JOrbs}, $N_M(B) \sim 3^4.(2 \times 2\times \Sym(4))$
which is not a subgroup of $L$. Therefore  $C_G(B) \not = C_{L}(B)$
and hence Theorem~\ref{PrinceThm} gives $C_G(B)/B \cong
\Aut(\U_4(2))$ or $\Sp_6(2)$.  By Lemma~\ref{New2} (i) $B \not \le C_{L^*}(B)'$ and $C_{L^*}(B)$ contains a Sylow $3$-subgroup of $C_G(B)$. Hence, by the Gasch\"utz Splitting Theorem, $E =O^3(C_G(B))\cong \Aut(\U_4(2))$ or $\Sp_6(2)$.
As $t$ inverts $B$, $N_G(B)/E \cong
\Sym(3)$. Since $t$ centralizes $E \cap J$ which is
elementary abelian of order $3^3$ and since this subgroup is
self-centralizing in $E$, we infer that $B \langle t \rangle =
C_{N_G(B)}(E) \cong \Sym(3)$. Thus the lemma will be proved once we
have eliminated the possibility that $E \cong \Sp_6(2)$.

Suppose that $E \cong \Sp_6(2)$. Then $E $ contains a subgroup $F$
with $F \cong \Sp_2(2) \times \Sp_4(2) \cong \Sym(3) \times
\Sym(6)$. Since there is a unique conjugacy class of elementary
abelian subgroups of order $27$ in $\Sp_6(2)$, we may choose $F$  so
that $J \cap E \in \Syl_3(F)$. Note that $t$ centralizes $F$. Let
$R_1 \in \syl_2 (N_{\langle t \rangle F}(J\cap E))$. Then $R_1 \cong
2\times 2\times \Dih(8) \le N_F(J)$ and $R_1$ contains $t$ which
inverts $B$. Let $x \in R_1^\prime$ be an involution. Then $x \in F'' \cong \Alt(6)$
and $x $ inverts $J \cap F''$ and centralizes $O_3(F)B$. On the
other hand, by Lemma~\ref{M0}, $R_1 \le M_1\sim 3^4.(2\times
\Sym(6))$ and so $R_1' \le M_*$. But then $C_J(x)$ contains
$3$-central elements of $G$  by Lemma~\ref{Modcalc} (ii). Hence $O_3(F)B$
contains a $3$-central element of $G$, say  $e$. However this means
that $\Alt(6) \cong F'' \le C_G(e) \sim
3^{1+4}_+.2^{1+4}_-.\Alt(5)$, which is absurd. Hence $E \not \cong
\Sp_6(2)$ and the lemma is proven.
\end{proof}

Now set $E=O^3(C_G(B))$, $K= C_G(t)$, $E_L = E \cap L$, $E_M = E\cap M$ and $J_K= J \cap K$.

 \begin{lemma}\label{ELEM}
$E_L \sim 3^{1+2}_+.\GL_2(3)$ and $E_M = N_E(J_K) \sim
3^{3}.(2\times \Sym(4))$.
 \end{lemma}

\begin{proof} We have that $E= C_G(\langle t,B\rangle)$ and so $Z$
and $J_K$ are contained in $E$. That $Z$ is a $3$-central subgroup
of $E$ follows from Lemma~\ref{Prince1}. Hence, as $E \cong\Aut(\U_4(2))\cong  \mathrm {PGSp}_4(3)$, we get  $E_L\sim
3^{1+2}_+.\GL_2(3)$ and,  since a Sylow
$3$-subgroup of $E$ contains a unique elementary abelian subgroup of
order $27$,  $E_M = N_E(J_K) \sim 3^{3}.(2\times
\Sym(4))$ (see for example \cite[pg. 26]{Atlas}).
\end{proof}

\section{The centralizer of $t$}

We now start our investigation of the centralizer of the involution
$t$. We contine with the notation of the last section. In particular, $K = C_G(t)$. By Lemma~\ref{Prince2}, $K$ contains
$E=O^3(C_G(B)) \cong \Aut(\U_4(2))$. Our first lemma asserts that we
already see the Sylow $3$-subgroup of $K$ in $C_L(t)$.

\begin{lemma}\label{CLtsylow} $C_S(t)$ is a Sylow $3$-subgroup of
$K$. In particular, $|K|_3=3^4$ and $E$ contains a Sylow
$3$-subgroup of $K$.
\end{lemma}

\begin{proof} Let $F = C_S(t)$. Then Lemmas~\ref{inv2}(iii) and
\ref{CStnew} imply that $Z(F)= Z$ and $F \in \Syl_3(C_{L}(t))$. If
$F_1 \in \syl_3(K)$ and  $F\le F_1$, then $N_{F_1}(F) $ normalizes
$Z$ and is consequently contained in $L$. Thus $N_{F_1}(F)=F$ and so
$F=F_1$.
\end{proof}

\begin{lemma}\label{utnotconj} The involutions  $t$ and $u$ are not $G$-conjugate and $u\in M_ *$.
\end{lemma}

\begin{proof} Choose an element $s$ of order $2$ in $N_{M_*}(S)$.
Then $s$ inverts $S/J$. Using Lemma~\ref{Modcalc}(ii) and (vi) we
see that $s$ centralizes $J/(J\cap Q)$ and $Z$, and inverts $(Q\cap
J)/Z$. Since $s$ normalizes $Q$ by Lemma~\ref{centralizers} (i), we
deduce that $\langle s\rangle Q= \langle u\rangle Q$. In particular,
$u \in M_*$ and so we have that $C_S(u)=C_J(u)$ contains exactly two
$3$-central subgroups by Lemma~\ref{Modcalc}(ii). Let $F=C_S(u)$.
Suppose that $F_1 \in \syl_3(C_G(u))$ with $F\le F_1$. If $F_1> F$,
then $|Z^{N_{F_1}(F)}| = 3$ which is not the case. Thus $F_1= F$ has
order $9$ and consequently, using Lemma~\ref{CLtsylow}, we see that
$t$ and $u$ are not $G$-conjugate.
\end{proof}

\begin{lemma}\label{fusinv1} Suppose that $x$ is an involution of $M$ with $|C_J(x)|= 3^3$. Then $x$ is $M$-conjugate to $t$.
\end{lemma}

\begin{proof} The involutions $xJ \in M/J$ with $|C_J(x)|=3^3$ are reflections on $J$. Since the two reflection classes are fused in $\CO_4^-(3)$, we have that all such involutions $xJ$ are conjugate. But then $M$ has exactly one class of such involutions.
\end{proof}

Recall that $J_K = J\cap K= C_J(t)$ is elementary abelian of order $3^3$.

\begin{lemma}\label{NJK} We have that \begin{enumerate}
\item $J_K = J(C_S(t))$;
\item $N_G(J_K) \le M$;
\item $C_G(J_K) = J\langle t\rangle$;
\item $N_K(J_K)/C_K(J_K)\cong \GO_3(3) \cong 2 \times \Sym(4)$; and
\item $N_K(J_K) \le \langle t \rangle E$
\end{enumerate}
\end{lemma}

\begin{proof} Since $C_S(t)$ is isomorphic to a Sylow $3$-subgroup
of $\PSp_4(3)$ by Lemma~\ref{inv2}(iii), (i) holds.

Let $Y = N_G(J_K)$. Then $Y$ normalizes $C_Y(J_K)$ which contains
$J$. Hence the Frattini Argument implies that $Y = N_Y(J)C_Y(J_K)$.
Since $Z \le J_K$, $C_G(J_K) \le L_*$. Because $J_K$ centralizes
$t$, $J_K \not\le Q$ and so $C_G(J_K) \le N_{L_*}(S)$ is $3$-closed.
It follows that $C_Y(J_K)$ normalizes $J=J(S)$. So (ii) holds.

Since $J_K= C_J(t)$, we have that $J_K$ and $[J,t]$ are orthogonal. In particular,
$N_K(J_K)/J_K\cong 2\times \GO_3(3) \cong 2\times 2 \times \Sym(4)$ and $C_G(J_K) = J\langle t\rangle$. This is (iii) and (iv). As
$N_G(J_K) \le N_M(B)$ part (iii) also holds.
\end{proof}

\begin{lemma}\label{S3S6} $K$ contains a subgroup isomorphic to $\Sym(3)\times
\Sym(6)$.
\end{lemma}

\begin{proof} Let $t_1 \in
E$ be such that $C_E(t_1) \cong 2 \times \Sym(6)$. Then $C_G(t_1)
\ge B\langle t\rangle \times C_E(t_1) \cong \Sym(3) \times 2 \times
\Sym(6)$ and so it suffices to show that $t$ and $t_1$ are
$G$-conjugate.

We make our initial choice of $t_1$ so that there exists  $F \in
\Syl_3(C_E(t_1))$ such that $F \le C_S(B)$. Then by Lemma~\ref{88}
$F$ is contained in the Thompson subgroup of $S\cap E$ which is
$J_K$. Hence $BF \le J$.

 Since $BF$ is a maximal
subgroup of $J$, $BF$ contains a conjugate of $Z$ by
Lemma~\ref{hyperplanes}. Conjugating by a suitable element of $M$ we
may then suppose that  $Z\le BF \le J$ and $t_1$ centralizes $BF$.
Thus we may view the entire configuration in $L_*$. By Lemma~\ref{inv1}(i),  $t_1 \in QP$. Therefore,
 either $t_1 $ is conjugate to $u$ or $t_1$ is conjugate to $t$.
Since $|C_{L_*}(u)|_3= 3^2$ by Lemma~\ref{inv1}(v), we have that $t_1$ is conjugate to $t$ as claimed.
\end{proof}


For $n \in \{0,1,2,3,4\}$, $\mathcal Z_n$ denotes the set of
subgroups of $J_K$ of order $9$ containing precisely $n$ subgroups
which are $G$-conjugate to $Z$.

\begin{lemma}\label{order9inpsp43}
\begin{enumerate}
\item $J_K$ contains exactly $4$ subgroups $G$-conjugate to
$Z$ and the remaining subgroups of $J_K$ of order $3$ are all
$G$-conjugate to $B$.

 \item The $N_K(J_K)$ orbits, under conjugation,  of the subgroups of
$J_K$ of order $9$  are $\mathcal Z_0$, $\mathcal Z_1$ and $\mathcal
Z_2$. Further, $|\mathcal Z_0|=3$, $|\mathcal Z_1|=4$ and $|\mathcal
Z_2|=6$.
\end{enumerate}
\end{lemma}

\begin{proof} From Lemma~\ref{NJK} (iii), we have
$N_K(J_K)/C_K(J_K) \cong \GO_3(3)\cong 2\times \Sym(4)$. Since $J_K$ is irreducible
as an $N_K(J_K)$-module, the centre of $N_K(J_K)/C_K(J_K)$ inverts
$J_K$ and thus has no effect on the orbits of $N_K(J_K)$ on
subgroups of $J_K$. Since $J_K$ can be identified as a non-degenerate orthogonal module and $N_K(J_K)/C_K(J_K)$ can be identified with $\GO_3(3)$, we see that $J_K$ has exactly four subgroups of order $3$ which correspond to singular one spaces and these are  $Z^{N_K(J_K)}$. The other subgroups of $J_K$ of order $3$ are conjugate to $B$.

When $N_K(J_K)$ acts on subgroups $A$ of order $9$ in $J_K$, we have three possibilities: $A$ could be hyperbolic, there are six of these, definite, there are three of these, or degenerate of which there are four.  By Witt's Lemma the respective types are fused in $N_K(J_K)$.  Therefore $\mathcal Z_0$ consists of definite spaces, $\mathcal Z_1$ of degenerate spaces and $\mathcal Z_2$ of hyperbolic spaces.
\end{proof}

\begin{lemma}\label{AZ1} Let $A \in \mathcal Z_1$ and $a\in A^\#$ be $3$-central. Then $A = J_K \cap
O_3(C_G(a))$.
\end{lemma}

\begin{proof} By Lemma~\ref{New2} (ii), we have that $J_K \cap
O_3(C_G(a))\in\mathcal Z_1$. The result is now verified as, by
Lemma~\ref{order9inpsp43}, there are exactly four
$N_K(J_K)$-conjugates of
 $\langle a\rangle$ in $J_K$  and $|\mathcal Z_1|=4$. \end{proof}

\begin{lemma}\label{JKsyl} Suppose that $|J_K:A|=3$. Then $J_K \in
\Syl_3(C_K(A))$. In particular, setting $E_b = O^3(C_G(b))$,  either
$C_{E_b}(t) \cong 2 \times \Sym(6)$ or $C_{E_b}(t) \sim
2^{1+4}_+.3^2.2^2$.
\end{lemma}

\begin{proof} Since $J_K$ is abelian and $J_K \le K$, $J_K \le C_K(A)$.
By Lemma~\ref{order9inpsp43}, there exists $b \in A$ which is not
$3$-central. Now $C_G(b) \cong 3 \times \Aut(\U_4(2))$ by
Lemma~\ref{Prince2}. Since $t$ centralizes $J_K \cap E_b$ which has order $9$, from
Table~\ref{TabInv} we read that $|C_{E_b}(t)|_3 = 3^2$. Now we may further deduce the possible structures of $C_{E_b}(t)$ as listed.
\end{proof}

\begin{lemma} \label{signal1}Suppose that $|J_K:A|=3$.
\begin{enumerate}
\item If $A\in \mathcal Z_1$ and $a\in A^\#$ is $3$-central, then
$O_{3^{\prime}}(C_K(A))=O_{3^{\prime}}(C_K(a)) \cong \Q(8)$. Also,  for $b \in A^\#$ with
$b$ not $3$-central in $G$,
$$O_{3^{\prime}}(C_K(A)) \le
O_{3^\prime}(C_K( b))\cong 2^{1+4}_+.$$
\item If $A\in \mathcal Z_0 \cup \mathcal Z_2$, then
$O_{3^{\prime}}(C_K(A)) =\langle t\rangle.$
\item If $T \in \signal_{C_G(A)}(J_K,3')$, then $T \le O_{3'}(C_K(
A))$.
\end{enumerate}
\end{lemma}

\begin{proof}

Assume that $A\in \mathcal Z_1$.  Let $a \in A^\#$ be a $3$-central
element and $b \in A\setminus \langle a \rangle$. Then $C_G(a) \sim
3^{1+4}_+.2^{1+4}_-.\Alt(5)$. Since every element of order $2$ in
$C_G(a)$ is contained in $O_{3,2}(C_G(a))$ by
Lemma~\ref{Laction2}(ii), we have that $t \in O_{3,2}(C_G(a))$. As
$t$ is not conjugate to the elements in $Z(C_G(a)/O_3(C_G(a)))$ by
Lemma~\ref{utnotconj}, we have $O_{3'}(C_{C_G(a)}(t)) \cong \Q(8)$
by Lemma~\ref{inv2}(ii). By Lemma~\ref{AZ1}, $A =J_K
\cap O_3(C_G(a))\le C_{O_3(C_G(a))}(t)$. Thus Lemma~\ref{inv2} (i)
and (ii) imply that $O_{3'}(C_K(a))=O_{3^{\prime}}(C_K(A)) \cong \Q(8)$ which is
the first claim in (i). We now focus on $b$. Using
Lemmas~\ref{order9inpsp43} (i) and \ref{Prince2}, we have $C_G(b)
\cong 3\times \Aut(\U_4(2))$. Let $E_b= O^3(C_G(b))$. Then, as $t$
centralizes $b$, $t\in E_b$. Now $C_{C_G(b)}(t)$ contains
$O_{3^{\prime}}(C_K( A)) \cong \Q(8)$. Hence, as $2\times \Sym(6)$
doesn't contain a subgroup isomorphic to $\Q(8)$, we may use
Lemma~\ref{JKsyl} to deduce that $t\in E_b'$ and that
$C_K(\langle b\rangle) \sim 3\times 2^{1+4}_+.3^2.2$. Thus
$O_{3^\prime}(C_K( b))\cong 2^{1+4}_+.$
Now using the fact that $C_K( b)$ is soluble and applying Lemma~\ref{const} we get that $O_{3'}(C_K(A)) \le O_3'(C_K( b))$. Thus (i) holds.

Assume that $A \in \mathcal Z_2$ and just as above let $a\in A^\#$
be a $3$-central element. By Lemma~\ref{AZ1}, $A \not \le
O_3(C_G(a))$. Let $b \in A\setminus O_3(C_G(a))$. Again by
Lemma~\ref{utnotconj}, $t$ is not conjugate to an element of the inverse image of
$Z(C_G(a)/O_3(C_G(a)))$. Hence using Lemmas~\ref{Laction2} (ii)
 and \ref{inv2}(iv) we get $C_{O_{3'}(C_K(a))}(b) = \langle t
\rangle$. In particular, using Lemma~\ref{const} again (ii) holds for $A\in \mathcal Z_2$.

Suppose that $A\in\mathcal Z_0$. Let $b\in A^\#$. Then  $C_G(b)
\cong 3 \times \Aut(\U_4(2))$ by Lemma~\ref{order9inpsp43} (i). Recall that
$E_b = O^3(C_G(b))$. Then from Lemma~\ref{JKsyl}, we have $C_{C_G(b)}(t) \sim 3
\times 2^{1+4}_+.3^2.2^2$ or $C_{C_G(b)}(t) \cong 3\times 2 \times
\Sym(6)$. In the latter case the centralizer in $C_{C_G(b)}(t)$ of
any further element of order $3$ has shape $2\times 3\times 3 \times
\Sym(3)$ and so (ii) holds if this possibility arises.  So assume
the former possibility occurs. Then, as $O^2(C_{E_b}(t))$ is
isomorphic to the central product $\SL_2(3)\circ\SL_2(3)$,
$C_{O_2(C_{E_b}(t))}(A\cap E_b)$ either has order $8$ or $2$. In the
former case we deduce from centralizer orders that $A\cap E$ is
$3$-central in $E$ and consequently $3$-central in $G$, a
contradiction. Thus $C_{O_2(C_{E}(t)) }(A\cap E)=\langle t\rangle $
and so (ii) holds when $A \in \mathcal Z_0$.

By Lemma~\ref{JKsyl}, $J_K \in \Syl_3(C_K(A))$ and so, as  $C_K(A)$
is soluble, $J_KO_{3'}(C_K(A)) = O_{3',3}(C_K(A))$. Therefore any
$3'$-subgroup of $C_K(A)$ which is normalized by $J_K$ centralizes
$J_K O_{3'}(C_G(A))/O_{3'}(C_G(A))$. Hence, as $C_K(A)$ is soluble,
(iii) follows from Lemma~\ref{const}.
\end{proof}

Define $R = \langle O_{3'}(C_K(A))\mid A\in\mathcal Z_1\rangle$.
Notice, that by Lemma~\ref{signal1} (i) and (ii), we also have that
$R = \langle O_{3'}(C_K(A))\mid |J_K:A|=3\rangle$.

\begin{lemma}\label{Jsigdone}  $R \cong 2^{1+8}_+$ and
$\signal_{K}^*(J_K,3')=\{R\}$.
\end{lemma}

\begin{proof} As $J_K \le C_K(A)$ for all $A\in\mathcal Z_1$, $R$ as defined is normalized by
$J_K$. Let $\mathcal Z_1=\{A_1,A_2,A_3,A_4\}$. Then, by
Lemma~\ref{signal1}(i), for $1\le i\le 4$, $O_{3'}(C_K(A_i)) \cong
\Q(8)$. Additionally, for $1\le i < j \le 4$, $A_i \cap A_j$ is a
$G$-conjugate of $B$ by Lemmas~\ref{order9inpsp43}(i) and \ref{AZ1}.
Thus $O_{2}(C_{K}(A_i\cap A_j)) \cong 2^{1+4}_+\cong
\Q(8)\circ\Q(8)$ by Lemma~\ref{signal1} (i). Note that $2^{1+4}_+$
contains exactly two subgroups isomorphic to $\Q(8)$ and that these
subgroups commute. Assume that $O_{3'}(C_K(A_i))= O_{3'}(C_K(A_j))$,
then this subgroup is centralized by $\langle A_i,A_j\rangle = J_K$.
Since $\mathcal Z_0\cup\mathcal Z_2\not=\emptyset$, this contradicts
Lemma~\ref{signal1} (ii) and (iii). Thus
$[O_{3'}(C_K(A_i)),O_{3'}(C_K(A_j))]=1$. It follows now that $R$ is
a central product of four subgroups each isomorphic to $\Q(8)$ and so
$R \cong 2^{1+8}_+$. In particular, $R \in \signal_K(J_K,3' )$.

Suppose that $R_0 \in \signal_K(J_K,3' )$. Then $R_0 = \langle
C_{R_0}(A)\mid |J_K:A|=3\rangle$. Since, for $|J_K:A|=3$,
$C_{R_0}(A) \in \signal_{C_G(A)}(J_K,3')$, we have $C_{R_0}(A) \le
O_{3'}(C_K( A ) )$ by Lemma~\ref{signal1} (iii). But then by
Lemma~\ref{signal1} (i) and (ii), $R_0 \le R$. Hence
$\signal_{K}^*(J_K,3')=\{R\}$.
\end{proof}

\begin{lemma}\label{signalmore} Suppose that $A \in \mathcal Z_1$. Then $$R=
\langle O_{3'}(C_K(b))\mid b \in A^\#, b \mbox{ not }
3\mbox{-central in } G\rangle.$$
\end{lemma}

\begin{proof} We have $C_R(A) \cong \Q(8)$ by
Lemma~\ref{signal1}(i). By Lemma~\ref{Jsigdone}, $R/C_R(A)$ is
elementary abelian of order $2^6$ and $O_{3'}(C_K(b)) \le R$. Since for $b\in A^\#$ such that
$b$ is not $3$-central in $G$, we have $|O_{3'}(C_K(b))/C_R(A)|=2^2$
by Lemma~\ref{signal1} (i), we infer that  $$R= \langle
O_{3'}(C_K(b))\mid b \in A^\#, b \mbox{ not } 3\mbox{-central in }
G\rangle.$$
\end{proof}

\begin{lemma}\label{Ein} $N_K(R)\ge RE$ and $C_K(R) =\langle t \rangle$.
\end{lemma}

\begin{proof} By Lemma~\ref{ELEM},  $E_L
\sim 3^{1+2}_+.\GL_2(3)$ and $E_M \sim 3^{3}.(2\times \Sym(4))$.
Furthermore, $O_3(E_M) = J_K$. Since $R$ is the unique member of
$\signal_{K}^*(J_K,3')$, $E_M$ normalizes $R$. Let $T = O_3(E_L)$.
Then $T\cap J_K\in \mathcal Z_1$ by Lemma~\ref{New2} (ii). Let $x
\in E_L \setminus E_M$ and set $A = (T \cap J_K)^x$. Note that $A
\le T^x= T$, so $A$ normalizes $R$ and $R^x$.  Now, using Lemma~\ref{signalmore} applied to the action of $A$ on $R^x$ we have, $$R^x=
\langle O_{3'}(C_K(b))\mid b \in A^\#, b \mbox{ not }3\mbox{-central in } G
\rangle.$$
Next we consider the action of $A$ on $R$. By coprime action we have $R= C_R(Z)\langle [C_R(b),C_R(Z)]\mid b \in A\setminus Z\rangle$. By  Lemmas~\ref{signal1} and \ref{signalmore},  $C_R(Z)= O_{3'}(C_K(Z))= O_{3'}(C_K(Z))^x= C_{R^x}(Z)$. Let $b \in A\setminus Z$. Then $b$ is not $3$-central in $G$ and consequently $C_R(b) \le C_K(b) $ which has shape $3\times 2^{1+4}_+.3^2.2$. Therefore any $2$-subgroup of $O^{3'}(C_K(b))$ is contained in $O_{3'}(C_K(b))$. Hence $[C_R(b),Z] \le O_{3'}(C_K(b))\le R^x$. It follows that $R \le R^x$ and so $R=R^x$.
 Thus $R$
is normalized by $\langle E_M,x\rangle = E$.

Let $C = C_K(R)$. Then, as $E$ contains a Sylow $3$-subgroup of $K$
by Lemma~\ref{CLtsylow}  and $E$ acts non-trivially on $R$, $C_K(R)$
is a $3'$-group which is normalized by $E$ and hence by $J_K$. Thus
$C_K(R) \le R$ by Lemma~\ref{Jsigdone}.
\end{proof}

We now set $H= N_G(R)$. Notice that as $R$ is extraspecial, we have
that $H$ centralizes $t$ and so $H= N_K(R)$. Our next goal is to
show that $G$, $H$ and $R$ satisfy the hypothesis of
Theorem~\ref{closed}.

\begin{lemma}\label{HmodR} $H/R \cong \Aut(\U_4(2))$ or $\Sp_6(2)$.
\end{lemma}

\begin{proof} We have that $Z \le E \le N_G(R)$ by Lemma~\ref{Ein}.
From the definition of $R$ and Lemma~\ref{inv2} (iii),
$O_2(C_{L_*}(t)) \le R$. Thus $C_{L_*}(t) R/R \cong
C_{L_*}(t)/O_{2}(C_{L_*}(t))$ is isomorphic to the centralizer of a
$3$-central element of order $3$ in $\PSp_4(3)$. Since $ER/R \ge
C_{L_*}(t) R/R$ we infer that $ZR/R$ is inverted by its normalizer
in $H/R$.  By Lemma~\ref{Jsigdone} the assumptions of  Theorem~\ref{PrinceThm} are fulfilled and we have $H/R \cong \Aut(\U_4(2))$ or
$\Sp_6(2)$.
\end{proof}

\begin{lemma} \label{part(ii)}$C_H(R) \le R$ and $R/\langle t\rangle$ is a minimal
normal subgroup of $H/\langle t\rangle$ of order $2^8$.
\end{lemma}

\begin{proof} Lemma~\ref{Ein} ensures that $C_H(R) \le R$. Also as $R$ is extraspecial of order $2^9$,
$R/\langle t \rangle$ has order $2^8$. Suppose that $R_1$ is a
 normal subgroup of $H$ contained in $R$ with $\langle t \rangle \le
 R_1\le R$. Now $J_KR/R$ is elementary abelian of order $27$ and the
 $3$-rank of $\GL_5(2)$ is $2$, and therefore  either $R/R_1$ or $R_1$
 is centralized by $O^{2}(H/R)$ and hence by $J_K$. However
 $C_{G}(J_K)= J\langle t \rangle$ by Lemma~\ref{NJK}(iii) and so we
 see that either $R=R_1$ or $R_1=\langle t \rangle$. Thus $R/\langle
 t  \rangle$ is a minimal normal subgroup of $H/\langle t\rangle$.
\end{proof}

\begin{lemma}\label{Kpfusion} The following hold.
\begin{enumerate}
\item $C_K(Z) \le H$.
\item $ER$ controls fusion of elements of order $3$ in $K$.
\item $B^G \cap K = B_1^K \cup B_2^K$ where $B_1$ is conjugate to a
subgroup of $J_K$ which together with $Z$ forms a subgroup in
$\mathcal Z_1$.
\item If $B_1 \le J_K$, then $C_K(B_1) \le ER$.
\end{enumerate}
\end{lemma}
\begin{proof}  Looking  in $E$, we see $C_E(Z) \sim
3^{1+2}.\SL_2(3)$. From Lemma~\ref{Jsigdone}, we have $C_R(Z) \cong
\Q(8)$. Since $|C_K(Z)|= 2^6.3^4$ by Lemma~\ref{inv2} (iii), part
(i) holds.

Since $J_K$ is torus in $E\cong \SU_4(2)$ (or using \cite[pg. 26]{Atlas}), we have that every element of order $3$
in $E$ is $E$-conjugate to an element of $J_K$. Since $E$ contains a
Sylow $3$-subgroup of $K$ and $N_K(J_K)$ controls $K$-fusion of
$3$-elements in $J_K$ by Lemma~\ref{fusion}, we have (ii).

As $J_K$ is an orthogonal module for $N_K(J_K)/C_K(J_K) \cong \GO_3(3)$, Lemma~\ref{NJK} (iv) implies $K$ has three
conjugacy classes of elements of order $3$ and just one $3$-central
class. Thus (iii) follows from (ii).

Now consider the class $B_1^K$. We may suppose that $B_1Z\in
\mathcal Z_1$. Then $C_R(B_1) \cong 2^{1+4}_+$ by
Lemma~\ref{signal1} (i). It follows that $t$ is an involution
contained in $O^{3}(C_G(B_1))' \cong \U_4(2)$ with $C_{C_G(B_1)}(t)
\sim 3\times 2^{1+4}_+.3^2.2^2 $. In particular, $(C_G(B_1) \cap
K)R/R$ normalizes $J_KR/R$ and so (iv) follows from Lemma~\ref{NJK}
(v).
\end{proof}

\begin{lemma}\label{fixpf} Continuing the notation of Lemma~\ref{Kpfusion}, we have cyclic groups in the same $H$-class as $B_2$  act fixed-point-freely on $R/\langle t
\rangle$.
\end{lemma}

\begin{proof}
Since $B_2$ is not contained in any member of $\mathcal Z_1$, we
have that $B_2$ acts faithfully on $O_{3'}(C_G(A))$ for each
$A\in\mathcal Z_1$. Thus, as  $R= \prod_{A \in\mathcal
Z_1}O_{3'}(C_G(A))$, we have that $B_2$ acts fixed-point-freely on
$R/ \langle t\rangle$.
\end{proof}

\begin{lemma}\label{not3embeddd} If $k \in K\setminus H$ and $d \in H \cap H^k$ has order $3$,
then $C_R(d) =\langle t\rangle$.
\end{lemma}

\begin{proof} We begin by noting that $R= O_2(H)$ and so $N_K(H)=
H$. Hence if there exists $k \in K\setminus H$, then $H \cap H^k\neq
H$

Suppose for a moment that a conjugate of $J_K $ is contained in $H
\cap H^k$. Then we may assume that $J_K \le H^k$. Thus $J_K$ and
$J_K^{k^{-1}}$ are both contained in $H$. Hence there exists $h \in
H$ such that $J_K= J_K^{k^{-1} h}$. But then $k^{-1}h \in N_K(J_K)
\le ER \le H$ by Lemmas~\ref{NJK} (v) and \ref{Ein}, whence $k \in
H$ and we have a contradiction.

Let $T \in \syl_3(H\cap H^k)$ and assume $T \neq 1$. Suppose that
$T$ contains a $K$-conjugate $Y$ of $Z$ or $B_1$. Then, as $H$
controls fusion of elements of order $3$ in $K$ by
Lemma~\ref{Kpfusion} (ii), we may suppose that either $Y=Z$ or
$Y=B_1$. Hence Lemma~\ref{Kpfusion} (i) and (iv) gives that $C_K(Y)
\le H$. However then $C_{H^k}(Y)$ contains a subgroup $X$ of $H^k$
which is conjugate to $J_K$ as every element of order $3$ in $H$ is
fused to an element of $J_K$ in $H$. But this means $ X\le C_K(Y)
\le H$ by Lemma~\ref{Kpfusion} (i) and (v) and this contradicts the
observation in paragraph two of the proof. It follows that if $d \in
H\cap H^k$ has order $3$ and $k \not \in H$, then $d$ is conjugate
to an element of $B_2$. The claim in the lemma now follows from
Lemma~\ref{fixpf}.
\end{proof}

\begin{proof}[Proof of Theorem~\ref{Th1}] Let $\ov K = K/\langle
t\rangle$ and set $\ov H = N_{\ov K}(\ov R)$. Lemmas~\ref{HmodR},
\ref{part(ii)}, \ref{Kpfusion} (ii) and \ref{not3embeddd} together
show that the hypotheses of Theorem~\ref{closed} are satisfied.
Therefore $\ov{K} = O_{2'}(\ov{K})\ov H$. Now $\ov{H}$ contains a
Sylow $3$-subgroup of $\ov {K}$ and so $O_{2'}(\ov K) \le
O_{3'}(\ov{K})$. Since $\signal_{K}^*(J_K,3')= \{R\}$, we infer that
$O_{2'}(\ov K)\le \ov {R}$. Thus $K=H$. Since, by Lemma~\ref{S3S6},
$K$ contains a subgroup isomorphic $\Sym(3) \times \Sym(6)$ whereas
$\Aut(\U_4(2))$ does not, we now get that $H/R\cong \Sp_6(2)$. Since
$O_3(G)= 1$, Lemma~\ref{JSig2} implies that $O_{2'}(G)=Z(G)=1$.
Since $R/\langle t \rangle$ is the spin-module for $H/R$,
Lemma~\ref{SP62facts}(iv) implies that the elements of order $5$ in
$H$ act fixed point freely on $R/\langle t\rangle$. Hence, at last,
Theorem~\ref{Smith} gives us that $G$ is isomorphic to $\Co_2$.
\end{proof}

\end{document}